\documentclass{amsart}
\usepackage{amssymb,color,xcolor}
\usepackage{eucal}
\usepackage{hyperref}
\usepackage{ifpdf}
\theoremstyle{plain}
\newtheorem{theorem}{Theorem}

\newtheorem{lemma}{Lemma}

\theoremstyle{definition}

\newtheorem{remark}{Remark}
\numberwithin{equation}{section}
\newcommand{\e}{\epsilon}
\newcommand{\s}{\sigma}

\newcommand{\R}{\mathbb R}
\newcommand{\N}{\mathbb N}

\newcommand{\C}{\mathbb C}

\newcommand{\Rn}{\mathbb R^n}

\newcommand{\rst}[1]{\ensuremath{{\mathbin\upharpoonright}%
\raise-.5ex\hbox{$#1$}}}

\def\p{\partial}

\def \e {\epsilon}

\def \s {\sigma}
\def \peso {e^{-\theta t^{-\sigma}}}
\def \pesodos {e^{-\frac{\theta}{2} t^{-\sigma}}}
\begin{document}

\title[Analyticity of solutions to parabolic evolutions and applications]{Analyticity of solutions to parabolic evolutions and applications}
\author{Luis Escauriaza}
\address[Luis Escauriaza]{Universidad del Pa{\'\i}s Vasco/Euskal Herriko Unibertsitatea\\Dpto. de Matem\'aticas\\Apto. 644, 48080 Bilbao, Spain.}
\email{luis.escauriaza@ehu.eus}
\thanks{The first two authors are supported  by Ministerio de Ciencia e Innovaci\'on grant MTM2014-53145-P. The last author is partially supported 
by FA9550-14-1-0214 of the EOARD-AFOSR and the National Natural Science Foundation of China under grants 11501424 and 11371285.}
\author{Santiago Montaner}
\address[Santiago Montaner]{Universidad del Pa{\'\i}s Vasco/Euskal Herriko
Unibertsitatea\\Dpto. de Matem{\'a}ticas\\Apto. 644, 48080 Bilbao, Spain.}
\email{santiago.montaner@ehu.eus}
\author{Can Zhang}
\address[Can Zhang]{School of Mathematics and Statistics, Wuhan University, 430072 Wuhan, China; Sorbonne Universit\'es, UPMC Univ. Paris 06, CNRS UMR 7598, Laboratoire Jacques- Louis Lions, F-75005 Paris, France.}
\email{zhangcansx@163.com}

\begin{abstract}
We find new quantitative estimates on the space-time analyticity of solutions to linear parabolic equations with analytic coefficients near the initial time. We apply the estimates  to obtain observability inequalities and null-controllability of parabolic evolutions over measurable sets.
\end{abstract}
\maketitle
\section{Introduction}
This work is concerned with the study of quantitative estimates up to the boundary of analyticity in the spatial and time variables of solutions to boundary value parabolic problems for small values of the time variable. If $\Omega\subset \R^n$ is a bounded domain, we obtain new quantitative estimates of analyticity for solutions of 
\begin{equation}\label{problemaparabolico}
\begin{cases}
\p_t u+\mathcal Lu=0,\ &\text{in}\ \Omega\times (0,1],\\
u=Du=\ldots=D^{m-1}u=0,\ &\text{in}\ \p\Omega \times (0,1].
\end{cases}
\end{equation}
 Throughout the work $\mathcal L$ is defined by
\begin{equation}\label{E: eltipodeoperadore}
\mathcal L= (-1)^m \sum_{|\alpha|\leq 2m} a_{\alpha}(x,t) \p_x^\alpha,
\end{equation}
where $\alpha=(\alpha_1,\ldots,\alpha_n)$ is in $\N^n $ and $|\alpha|=\alpha_1+\cdots+\alpha_n$;
the coefficients of $\mathcal L$ are bounded and satisfy a uniform parabolicity condition, i.e., there is $\varrho>0$ such that 
\begin{equation}\label{parabolicidad}
\begin{aligned}
&\sum_{|\alpha|= 2m}  a_\alpha(x,t)\xi^{\alpha}\ge\varrho |\xi|^{2m},\ \text{for}\ \xi\in\R^n,\ (x,t)\in \Omega\times[0,1],\\
&\sum_{|\alpha|\le 2m}\|a_\alpha\|_{L^\infty(\Omega\times [0,1])}\le\varrho^{-1}.
\end{aligned}
\end{equation}

Our approach to prove quantitative estimates of analyticity is based on an induction process which employs $W_2^{2,1}(\Omega\times[0,1])$ Schauder estimates for solutions to parabolic initial-boundary value problems. This estimates were first derived in \cite{Solonnikov} for parabolic problems with quite general boundary conditions. In order to employ these estimates, we must assume that  $\partial\Omega$ is globally of class $C^{2m-1,1}$. Thus, the $W^{2m,1}_2(\Omega\times [0,1])$ Schauder estimates hold \cite[Theorem 6]{DongKim}; i.e., there is $K>0$ such that
\begin{equation}\label{E: Schauder estimate}
\|\partial_t u\|_{L^2(\Omega\times (0,1))}+\sum_{|\alpha|\le 2m}\|\partial^\alpha_xu\|_{L^2(\Omega\times (0,1))}\le K \left[\|F\|_{L^2(\Omega\times (0,1))} + \|u\|_{L^2(\Omega\times (0,1))}\right],
\end{equation}
when $u$ satisfies
\begin{equation*}
\begin{cases}
\p_t u+\mathcal Lu=F,\ &\text{in}\ \Omega\times (0,1],\\
u=Du=\ldots=D^{m-1}u=0,\ &\text{in}\ \p\Omega \times (0,1],\\
u(0)=0,\ &\text{in}\ \Omega.
\end{cases}
\end{equation*}

In this setting, we improve the quantitative estimates on the space-time analyticity of solutions to \eqref{problemaparabolico} available in the literature when the coefficients of $\mathcal L$ and the boundary of $\Omega$ are analytic. As far as we understand, the best quantitative bounds that we can infer or derive for solutions to  \eqref{problemaparabolico} from the reasonings in \cite{Friedman3,Friedman2, Friedman1,Tanabe,KinderlehrerNirenberg,Komatsu,Takac0,Takac}\color{black}\footnote{We explain in Section \ref{comment} our understanding of previous results.}\color{black}
are the following:

\vspace{0.1cm}
{\it There is $0<\rho\le 1$, $\rho=\rho(\varrho, m, n,\p \Omega)$ such that for $(x,t)$ in $\overline\Omega\times (0,1]$, $\alpha\in\N^n$ and $p\in\N$,

\begin{equation}\label{E: kimalito}
|\partial_x^\alpha\partial_t^p u(x,t)|\le \rho^{-1-\frac{|\alpha|}{2m}-p}\left(|\alpha|+p\right)!\, t^{-\frac{|\alpha|}{2m}-p-\frac{n+2m}{4m}}\|u\|_{L^2(\Omega\times(0,1))},
\end{equation}
where $2m$ is the order of the evolution and $|\alpha|=\alpha_1+\dots+\alpha_n$.}
\vspace{0.1cm}

The later can be seen to hold when the boundary of $\Omega$ is analytic and the coefficients of the underlying linear parabolic equation satisfy for some $0<\varrho\le 1$ bounds like
\begin{equation*}\label{E condiciongeneral}
|\partial_x^\alpha\partial_t^p A(x,t)|\le {\varrho}^{-1-|\alpha|-p}\left(|\alpha|+p\right)!,\ \text{for all}\ (x,t)\in\overline\Omega\times [0,1],\ \alpha\in\N^n\ \text{and}\ p\in\N.
\end{equation*}

A first observation regarding \eqref{E: kimalito} is that it blows up as $t$ tends to zero, something unavoidable since it holds for arbitrary $L^2(\Omega)$ initial data; however, \eqref{E: kimalito} provides a lower bound $\sqrt[2m]{\rho t}$ for the radius of convergence of the Taylor series in the spatial variables around any point in $\overline\Omega$ of the solution $u(\cdot,t)$ at times $0<t\le 1$. This lower bound shrinks to zero as $t$ tends to zero and does not reflect the \textit{infinite speed of propagation} of parabolic evolutions. Thus, it would be desirable to prove a quantitative estimate of space-time analyticity which provides a positive lower bound of the spatial radius of convergence for small values of $t$.

Concerning this and with the purpose to prove the interior and boundary null controllability of parabolic evolutions with \emph{time-independent} analytic coefficients over bounded analytic domains and with bounded controls acting over measurable sets of positive measure, we derived in \cite{AE,AEWZ,EMZ} the following quantitative estimates on the space-time analyticity of the solutions of such parabolic evolutions: there is $0<\rho\le 1$ such that for $(x,t)$ in $\overline\Omega\times (0,1]$, $\alpha\in\N^n$ and $p\in\N$,
\begin{equation}\label{estimatescontrol}
|\p_x^\alpha \p_t^p u(x,t)|\le  e^{1/\rho t^{1/(2m-1)}}\rho^{-|\alpha|-p} t^{-p}  \left(|\alpha|+p\right)! \|u\|_{L^2(\Omega\times(0,1))}.
\end{equation}
This was done by quantifying each step in a reasoning developed in \cite{LandisOleinik}, which reduces the study of the strong unique continuation property within characteristic hyperplanes for solutions of {\it time-independent} parabolic evolutions to its elliptic counterpart.

The bound \eqref{estimatescontrol} shows that the space-time Taylor series expansion of solutions converges absolutely over $B_\rho(x)\times ((1-\rho)t,(1+\rho)t)$, for some $0<\rho\le 1$, when $(x,t)$ is in $\overline\Omega\times (0,1]$. The later is an essential feature for its applications to the null-controllability of parabolic evolutions over measurable sets, while \eqref{E: kimalito} is not appropriate for such purpose \cite{AE,AEWZ,PW1,W1,C1}. Nevertheless, the reasonings leading to \eqref{estimatescontrol} in \cite{EMZ} can not be extended  to {\it time-dependent} parabolic evolutions.  Also, one can use upper bounds of the holomorphic extension to $\C^n$ of the fundamental solution of higher order parabolic equations or systems with constant coefficients \cite[p. 15 (15);  pp. 47-48 Theorem 1.1 (3)]{Eidelman} and Cauchy's theorem for the representation of derivatives of holomorphic functions as path integrals, to show that there is $\rho=\rho(n,m)$, $0<\rho\le 1$, such that the solution to
\begin{equation*}
\begin{cases}
\partial_t u+(-\Delta)^m u=0,\ &\text{in}\ \Rn\times (0,+\infty),\\
u(0)=u_0,\ &\text{in}\ \Rn,
\end{cases}
\end{equation*}
satisfies
\begin{equation}\label{E: quecachondeo}
|\partial_x^\alpha\partial_t^pu(x,t)|\le \rho^{-1-\frac{|\alpha|}{2m}-p}|\alpha|!^{\frac 1{2m}}\,p!\, t^{-\frac{|\alpha|}{2m}-p-\frac{n}{4m}}\|u_0\|_{L^2(\Rn)},
\end{equation}
when $\alpha\in\N^n$ and $p\in\N$. Thus, the radius of  convergence of the Taylor series expansion  of $u(\cdot, t)$ around points in $\Rn$ is $+\infty$ at all times $t>0$. The same holds when $(-\Delta)^m$ is replaced by other elliptic operators or systems of order $2m$ with constant coefficients. Also, observe that \eqref{estimatescontrol} is somehow in between \eqref{E: kimalito} and \eqref{E: quecachondeo}, since
\begin{equation*}
{t^{-\frac{|\alpha|}{2m}}} \lesssim |\alpha|!^{1-\frac{1}{2m}}e^{1/\rho t^{1/(2m-1)}},\ \text{for}\ \alpha\in\N^n,\ t>0.
\end{equation*}
Here, we derive a formal proof of \eqref{estimatescontrol} valid for all parabolic operators. To carry it out we quantify by induction the growth of each derivative $\p_x^\alpha \p_t^p$ of a solution to  \eqref{problemaparabolico} with respect to the weight function
\begin{equation*}
t^pe^{-\theta t^{-\s}},\ \text{with}\ 0\le\theta\le 1\ \text{and}\ \s=\tfrac1{2m-1}\, ,
\end{equation*}
over $\overline\Omega\times [0,1]$. To accomplish it we use the the full time interval of existence of the solutions before time $t$, the $W^{2m,1}_2$ Schauder estimate \eqref{E: Schauder estimate}, the weighted $L^2$ estimates in Lemmas \ref{Lema:regularidad1},  \ref{Lema:regularidad3} and \ref{Lema:regularidad4} and the inequalities \eqref{unadelasclaves}. We mention that the precise behavior of the bounds in \eqref{unadelasclaves} is key for our reasonings. The novelty of our proof rests on the fact that we use the weighted $L^2$ estimates in Lemmas \ref{Lema:regularidad1},  \ref{Lema:regularidad3} and \ref{Lema:regularidad4}. 

Throughout the work $\nu$  is the  exterior unit normal to the boundary of $\Omega$, $d\s$ denotes surface measure on $\partial\Omega$, $B_R$ stands for the open ball of radius $R$ centered at $0$ and $B_R^+=B_R\cap\{x_n> 0\}$. To describe the analyticity of a piece of  boundary  $B_R(q_0)\cap\p \Omega$ with $q_0$ in $\partial\Omega$,  we assume that for each $q$ in $B_R(q_0)\cap\partial\Omega $ we can find, after a translation and rotation, a new coordinate system (in which $q=0$) and an analytic function $\varphi : B_{\varrho}'\subset \R^{n-1}\longrightarrow\R$ verifying
\begin{equation}\label{E:descripcionfrontera}
\begin{aligned}
\varphi(0'&)=0,\ |\partial_{x'}^{\alpha}\varphi(x')|\le |\alpha|! \,\varrho^{-|\alpha|-1}\, ,\ \text{when}\ x'\in B_{\varrho}',\ \alpha\in \N^{n-1},\\
&B_{\varrho}\cap\Omega=B_\varrho\cap\{(x',x_n): x'\in B_{\varrho}',\  x_n>\varphi(x')\},\\
&B_{\varrho}\cap\partial\Omega=B_\varrho\cap\{(x',x_n): x'\in B_{\varrho}',\ x_n=\varphi(x')\},
\end{aligned}
\end{equation}
where $B'_{\varrho}=\{x' \in \R^{n-1},\ |x'|<\varrho\}$. Regarding the analytic regularity of the coefficients, we consider the following conditions:

\vspace{0.1cm}
{\it Let $x_0$ in $\overline{\Omega}$, there is $\varrho>0$ such that for any $\alpha\in\N^n$ and $p\in\N$,
\begin{align}
|\partial_{x}^\gamma \partial_t^p a_\alpha(x,t)|&\leq  \varrho^{-1-|\gamma|-p}|\gamma|!p!,\ \ \text{in}\ B_R(x_0)\cap\overline{\Omega}\times [0,1], \label{coeff:space_analytic}\\
|\partial_t^p a_\alpha(x,t)|&\leq  \varrho^{-1-p}p!,\ \ \text{in } \ \overline{\Omega}\times [0,1]. \label{coeff:time_analytic}
\end{align}}

The main result in this work is the following:
\begin{theorem}\label{lagrananaliticidad} 
Let $x_0$ be in $\overline{\Omega}$, $0<R\le 1$. Assume that $\mathcal L$ satisfies \eqref{parabolicidad}, \eqref{coeff:space_analytic}, \eqref{coeff:time_analytic} and $B_R(x_0)\cap\p\Omega$ is analytic when it is non-empty. Then, there is $\rho=\rho(\varrho,m,n)$, $0<\rho\leq 1$, such that the inequality
\begin{equation} \label{todanaliticidad}
|\p_x^\alpha \p_t^p u(x,t)|\leq e^{1/\rho t^{1/(2m-1)}}  \rho^{-1-|\alpha|-p}R^{-|\alpha|} t^{-p}\left(|\alpha|+p\right)!\,  \|u\|_{L^2(\Omega\times(0,1))},
\end{equation}
holds   for all $\alpha\in \N^n$, $p\in \N$ and $(x,t)\in B_{\frac R2}(x_0)\cap \overline{\Omega}\times (0,1]$, when $u$ solves \eqref{problemaparabolico}. 
\end{theorem}
\begin{remark}\label{R:1}
If we only assume that the coefficients of $\mathcal L$ are measurable in the time variable, satisfy \eqref{coeff:space_analytic} for $p=0$ and $B_R(x_0)\subset \Omega$, then  \eqref{todanaliticidad} holds in $B_{R/2}(x_0)\times(0,1]$ with $p=0$. This follows from Remark \ref{R: unacosiat}.

If we only assume \eqref{coeff:space_analytic} for some $x_0$ in $\overline\Omega$, so that some of the coefficients of $\mathcal L$ may not be globally analytic in the time-variable over $\overline\Omega$, the solutions of \eqref{problemaparabolico} are still analytic in the spatial variable over $B_{\frac R2}(x_0)\cap\overline\Omega\times [0,1]$ with  a lower bound on the radius of analyticity independent of time but only Gevrey of class $2m$ in the time-variable; i.e.,
\begin{equation*}
|\p_x^\alpha \p_t^p u(x,t)|\le e^{1/\rho t^{1/(2m-1)}}  \rho^{-1-|\alpha|-p} R^{-|\alpha|} \left(|\alpha|+2mp\right)!\,  \|u\|_{L^2(\Omega\times(0,1))}.
\end{equation*}
when $(x,t)\in B_{\frac R2}(x_0)\cap\overline\Omega\times (0,1]$, $\alpha\in \N^n$ and $p\in \N$.

At the end of Section \ref{main_analyticity} we give a counterexample showing that solutions can fail to be time-analytic at all points of a hyperplane $\Omega\times\{t_0\}$, when some of the coefficients are not time-analytic in a proper subdomain $\Omega'\times\{t_0\}\subset  \Omega\times\{t_0\}$, $t_0>0$. Thus,  the lack of time-analyticity of the coefficients in a subset of a characteristic hyperplane $t=t_0$ can propagate to the whole hyperplane $t=t_0$. 
\end{remark}

Our motivation to prove Theorem \ref{lagrananaliticidad} stems from its applications to the null-controllability of parabolic evolutions with bounded controls acting over measurable sets of positive measure. 
The main tool used to establish null-controllability properties of parabolic evolutions are the observability inequalities from which the null-controllability follows by duality arguments \cite{DRu,Lions1}. The reasonings in \cite{AE,AEWZ,EMZ,PW1,W1,C1} make it now clear, that after Theorem \ref{lagrananaliticidad} is established, most of the results in those works can now be extended to parabolic evolutions with {\it time-dependent} coefficients and for {\it general} measurable sets with positive measure. We remark that only \cite{PW1} and \cite{PWZ} have dealt with some operators with {\it time-dependent} coefficients and measurable control regions but only for the special case of $\partial_t-\Delta+c(x,t)$, with $c$ bounded in $\R^{n+1}$ and for control regions of the form $\omega\times E$, with $\omega\subset\Omega$ an open set and $E\subset [0,T]$ a measurable set. In particular, \cite[Theorem 1 and \S 5]{EMZ} and Theorem \ref{lagrananaliticidad} imply Theorem \ref{th:observability1}, where $\mathcal L^\ast$ is the adjoint operator of $\mathcal L$. 
\begin{theorem}\label{th:observability1}
Let $0<T\le 1$, $\Omega$ be a bounded open set in $\Rn$ with analytic boundary, $\mathcal{D}\subset\Omega\times(0,T)$ be a measurable set with positive measure and $\mathcal L$ satisfy \eqref{coeff:space_analytic} over $\overline{\Omega}\times[0,1]$. Then, there is $N=N(\Omega,T,\mathcal{D},\varrho)$ such that the inequality
\begin{equation*}
\|\varphi(0)\|_{L^2(\Omega)}\leq
N\|\varphi\|_{L^1(\mathcal{D})}
\end{equation*}
holds for all $\varphi$ satisfying
\begin{equation*}
\begin{cases}
-\p_t \varphi+\mathcal L^\ast\varphi=0,\ &\text{in}\ \Omega\times [0,T),\\
\varphi=D\varphi=\ldots=D^{m-1}\varphi=0,\ &\text{in}\ \p\Omega \times [0,T),\\
\varphi(T)=\varphi_T,\ &\text{in}\ \Omega,
\end{cases}
\end{equation*}
with $\varphi_T$ in $L^2(\Omega)$. For each $u_0$ in $L^2(\Omega)$, there is $f$ in $L^\infty(\mathcal D)$ with
\begin{equation*}
\|f\|_{L^\infty(\mathcal D)}\le N\|u_0\|_{L^2(\Omega)},
\end{equation*}
such that the solution to
\begin{equation*}
\begin{cases}
\p_t u+\mathcal Lu=f\chi_\mathcal D,\ &\text{in}\ \Omega\times (0,T],\\
u=Du=\ldots=D^{m-1}u=0,\ &\text{in}\ \p\Omega \times (0,T],\\
u(0)=u_0,\ &\text{in}\ \Omega,
\end{cases}
\end{equation*}
satisfies $u(T)\equiv 0$. Also, the control $f$ with minimal $L^\infty(\mathcal D)$-norm  is unique and has the bang-bang property; i.e., $|f(x,t)|=$ const. for a.e. $(x,t)$ in $\mathcal D$. 
\end{theorem}
\begin{remark}\label{R:2}
When $\mathcal{D}=\omega\times(0,T)$, the constant in Theorem \ref{th:observability1} is of the form $e^{C/T^{1/(2m-1)}}$, with $C=C(\Omega,|\omega|,\varrho)$. 
\end{remark}
\begin{remark}
The proof of Theorem \ref{th:observability1} is the same as in \cite[Theorem 1 and \S 5]{EMZ} and requires energy estimates; i.e.,  we need to be able to solve the initial value problem
\begin{equation*}\label{problemaparabolico2}
\begin{cases}
\p_t u+\mathcal Lu=0,\ &\text{in}\ \Omega\times (0,1],\\
u=Du=\ldots=D^{m-1}u=0,\ &\text{in}\ \p\Omega \times (0,1],\\
u(0)=u_0,\ \text{in}\ L^2(\Omega),
\end{cases}
\end{equation*}
with data $u_0$ in $L^2(\Omega)$ and with a unique solution $u$ in the energy class $$C([0,1],L^2(\Omega))\cap L^2([0,1],H^m_0(\Omega)).$$
To make sure that such energy estimates and uniqueness of solutions hold in the later class, we recall that operators $\mathcal L$ as in \eqref{E: eltipodeoperadore}, which satisfy the conditions in Theorem \ref{th:observability1}, can always be written in variational form as
\begin{equation}\label{E: selfadjoint}
\sum_{|\alpha|\le 2m} a_\alpha(x,t)\partial^{\alpha}_x=\sum_{|\gamma|,|\beta|\le m}\partial^\gamma_x\left(A_{\gamma\beta}(x,t)\partial^\beta_x\ \right),
\end{equation}
with
\begin{equation}\label{parabolicidad1}
\begin{aligned}
&\sum_{|\gamma|=|\beta|=m} A_{\gamma\beta}(x,t)\xi^{\gamma}\xi^{\beta}\ge \varrho |\xi|^{2m},\ \text{for}\ \xi\in\R^n,\ (x,t)\in \Omega\times[0,1],\\
&\sum_{|\gamma|,|\beta|\le m}\|A_{\gamma\beta}\|_{L^\infty(\Omega\times [0,1])}\le\varrho^{-1},
\end{aligned}
\end{equation} 
for some possibly smaller $\varrho>0$. The later is claimed without a proof in \cite[p. 32]{Friedman1}. For the convenience of the reader we add its proof at the end of the Appendix.
\end{remark}
Similar results on boundary null-controllability over measurable sets with positive measure can be stated for higher order {\it time-dependent} parabolic evolutions, under the same global analyticity conditions as in Theorem \ref{th:observability1}.  This follows follows from Theorem \ref{lagrananaliticidad} and the reasonings in \cite[Theorem 2 and \S 5]{EMZ}.

For second order parabolic equations the last results hold with less global regularity assumptions on the coefficients and of the boundary of $\Omega$. In particular, from Theorem \ref{lagrananaliticidad}, \cite{FursikovOImanuvilov,ImanuvilovYamamoto2} and the telescoping series method one can get the following results for {\it time-dependent} second order parabolic equations 
\begin{equation*}
\partial_t-\nabla\cdot\left(\mathbf A(x,t)\nabla\ \right)+\mathbf b_1(x,t)\cdot\nabla+ \nabla\cdot\left(\mathbf b_2(x,t)\  \right)+c(x,t),
\end{equation*}
verifying
\begin{equation}\label{E:1}
\begin{split}
&\varrho \mathbf I\le \mathbf A\le\varrho^{-1}\mathbf I,\ \text{in}\ \Omega\times [0,1]\\
&\|\nabla_{x,t}\mathbf A\|_{L^\infty(\Omega\times [0,1])}+\max_{i=1,2}\|\mathbf b_i\|_{L^\infty(\Omega\times [0,1])} +\|c\|_{L^\infty(\Omega\times [0,1])}\le\varrho^{-1}.
\end{split}
\end{equation}
\begin{theorem}\label{T:otroteoremilla}
Let $0<T\le 1$, $\mathcal D\subset B_R(x_0)\times (0,T)$ be a measurable set with positive measure, $\Omega$ be a bounded  $C^{1,1}$ domain, $B_{2R}(x_0)\subset\Omega$, $\mathbf A$, $\mathbf b_i$, $i=1,2$ and $c$ also satisfy \eqref{coeff:space_analytic} over $B_{2R}(x_0)\times [0,1]$ and \eqref{coeff:time_analytic}. Then, there is $N=N(\Omega, T,\mathcal D, \varrho)$ such that the inequality
\begin{equation*}
\|\varphi(0)\|_{L^2(\Omega)}\le N\|\varphi\|_{L^1(\mathcal D)},
\end{equation*}
holds for all $\varphi$ satisfying
\begin{equation*}
\begin{cases}
-\p_t\varphi-\nabla\cdot\left(\mathbf A\nabla \varphi\right)-\nabla\cdot\left(\mathbf b_1\varphi\right)-\mathbf b_2\nabla\cdot\varphi+c\varphi=0,\ &\text{in}\ \Omega\times [0,T),\\
\varphi=0,\ &\text{in}\ \p\Omega \times [0,T],\\
\varphi(T)=\varphi_T,\ &\text{in}\ \Omega,
\end{cases}
\end{equation*}
for some $\varphi_T$ in $L^2(\Omega)$. For each $u_0$ in $L^2(\Omega)$, there is $f$ in $L^\infty(\mathcal D)$ with
\begin{equation*}
\|f\|_{L^\infty(\mathcal D)}\le N\|u_0\|_{L^2(\Omega)},
\end{equation*}
 such that the solution to
\begin{equation*}
\begin{cases}
\p_tu-\nabla\cdot\left(\mathbf A\nabla u\right)+\mathbf b_1\cdot\nabla u+\nabla\cdot\left(\mathbf b_2u\right)+cu=f\chi_{\mathcal D},\ &\text{in}\ \Omega\times (0,T],\\
u=0,\ &\text{in}\ \p\Omega \times [0,T],\\
u(0)=u_0,\ &\text{in}\ \Omega,
\end{cases}
\end{equation*}
satisfies $u(T)\equiv 0$. Also, the control $f$ with minimal $L^\infty(\mathcal D)$-norm  is unique and has the bang-bang property; i.e., $|f(x,t)|=$ const. for a.e. $(x,t)$ in $\mathcal D$. 
\end{theorem}
\begin{remark}\label{R: elultimo}
Theorem \ref{T:otroteoremilla} also holds when $\mathbf b_1\equiv 0$, \eqref{E:1} holds and $\mathbf A$, $\mathbf b_2$ and $c$ are only analytic with respect to the space-variables over $B_{2R}(x_0)\times [0,1]$. This follows from Remark  \ref{R:1} and the reasonings in \cite{PW1,PWZ,EFV}. We outline the proof  of this result in Section \ref{main_observability}.
\end{remark}
\begin{theorem}\label{T:otroteoremilla2}
Let $\Omega$ and $T$ be as above, $\mathcal J\subset\triangle_R(q_0)\times (0,T)$ be a measurable set with positive measure, $q_0\in\partial\Omega$, $\triangle_{2R}(q_0)$ be analytic, $\mathbf A$, $\mathbf b_i$, $i=1,2$ and $c$ also satisfy \eqref{coeff:space_analytic} over $B_{2R}(q_0)\cap\overline\Omega\times [0,1]$ and \eqref{coeff:time_analytic}. Then, there is $N=N(\Omega, T,\mathcal J,\varrho)$ such that the inequality
\begin{equation*}
\|\varphi(0)\|_{L^2(\Omega)}\le N\|\mathbf A\nabla\varphi\cdot\nu\|_{L^1(\mathcal J)},
\end{equation*}
holds for all $\varphi$ satisfying
\begin{equation*}
\begin{cases}
-\p_t\varphi-\nabla\cdot\left(\mathbf A\nabla \varphi\right)-\nabla\cdot\left(\mathbf b_1\varphi\right)-\mathbf b_2\nabla\cdot\varphi+c\varphi=0,\ &\text{in}\ \Omega\times [0,T),\\
\varphi=0,\ &\text{in}\ \p\Omega \times [0,T],\\
\varphi(T)=\varphi_T,\ &\text{in}\ \Omega,
\end{cases}
\end{equation*}
for some $\varphi_T$ in $L^2(\Omega)$. For each $u_0$ in $L^2(\Omega)$, there is $g$ in $L^\infty(\mathcal J)$ with
\begin{equation*}
\|g\|_{L^\infty(\mathcal J)}\le N\|u_0\|_{L^2(\Omega)},
\end{equation*}
 such that the solution to
\begin{equation*}
\begin{cases}
\p_tu-\nabla\cdot\left(\mathbf A\nabla u\right)+\mathbf b_1\cdot\nabla u+\nabla\cdot\left(\mathbf b_2u\right)+cu=0,\ &\text{in}\ \Omega\times (0,T],\\
u=g\chi_\mathcal{J},\ &\text{in}\ \p\Omega \times [0,T],\\
u(0)=u_0,\ &\text{in}\ \Omega,
\end{cases}
\end{equation*}
satisfies $u(T)\equiv 0$. Also, the control $g$ with minimal $L^\infty(\mathcal J)$-norm  is unique and has the bang-bang property; i.e., $|g(q,t)|=$ const. for a.e. $(q,t)$ in $\mathcal J$. 
\end{theorem}

As in \cite{AE,AEWZ,EMZ}, the main tools to derive these results are Theorem \ref{lagrananaliticidad}, the telescoping series method \cite{Miller2} and  Lemma \ref{propagation} below.  Lemma \ref{propagation} was first derived in \cite{Vessella}. See also \cite{Nadirashvili2} and \cite{Nadirashvili} for close results. The reader can find a simpler proof of  Lemma \ref{propagation} in \cite[\S 3]{AE}. The proof there is  built with ideas from \cite{Malinnikova}, \cite{Nadirashvili2} and  \cite{Vessella}.
\begin{lemma}\label{propagation}
Let $\omega\subset B_R$ be a measurable set, $|\omega |\ge\varrho |B_R|$, $f$ be an analytic function in $B_{R}$ and assume there are $M>0$ and $0<\varrho\le 1$ such that 
\begin{equation*}
|\p_x^\alpha f(x)|\leq M(R\varrho)^{-|\alpha|}|\alpha|!,\ \text{when} \ x\in B_{R}\ \text{and}\ \alpha \in \N^n.
\end{equation*}
Then, there are $N=N(\varrho)$ and $\theta=\theta(\varrho)$, $0<\theta<1$, such that
\begin{equation*}
\|f\|_{L^\infty(B_R)}\leq NM^{1-\theta}\left(\frac{1}{|\omega|}\int_{\omega}|f|dx \right)^{\theta}.
\end{equation*}
\end{lemma}
The paper is organized as follows: in Section \ref{main_analyticity} we prove Theorem \ref{lagrananaliticidad}, give an outline for the proof of the second part of Remark \ref{R:1} after Remark \ref{R: alguno} and then, finish Section \ref{main_analyticity} with the counterexample. Section \ref{main_observability} contains the proofs of Theorems \ref{T:otroteoremilla} and \ref{T:otroteoremilla2} and Remark \ref{R: elultimo}.
Section \ref{comment} provides a historical background on previous works.
 Section \ref{S:apendice} is an appendix which contains some Lemmas we use in Section \ref{main_analyticity}. 
 
\section{Proof of Theorem \ref{lagrananaliticidad}}\label{main_analyticity}
We prove Theorem \ref{lagrananaliticidad} in several steps following the scheme devised in \cite[Ch. 3, \S 3]{Friedman1}. Throughout the work $N$ denotes a constant depending on $\varrho$, $n$, $m$ and $R$. We also define
\begin{equation*}
\sigma=1/(2m-1),\  b=(2m-1)/2m,
\end{equation*}
\begin{equation*}
\|\cdot\|=\|\cdot\|_{L^2(\Omega\times (0,1))},\ \|\cdot\|_r=\|\cdot\|_{L^2(B_r\times (0,1))}\ \text{and}\ \|\cdot\|_r'=\|\cdot\|_{L^2(B_r^+\times (0,1))}  
\end{equation*}
with $B_r^+=\{x\in B_r : x_n>0\}$. We first prove an estimate related to the time-analyticity of global solutions.
\begin{lemma}\label{Th:time_analyticity} Assume that $\mathcal L$ satisfies \eqref{parabolicidad} and \eqref{coeff:time_analytic}. Then, there are $M=M(\varrho,n,m)$ and $\rho=\rho(\varrho,n,m)$, $0<\rho\le1$, such that 
\begin{equation}\label{time_analy}
\|t^{p+1} \partial_t^{p+1} u \|+\sum_{l=0}^{2m} \|t^{p+\frac{l}{2m}}D^l \p_t^{p}u \| \leq M \rho^{-p}(p+1)!\|u\|,
\end{equation}
holds for $p\in \N$ and all solutions $u$ to \eqref{problemaparabolico}.
\end{lemma}
\begin{proof} We prove \eqref{time_analy} by induction on $p$. For the case $p=0$ of \eqref{time_analy}, apply the weighted $L^2$ estimate in Lemma \ref{Lema:regularidad1} with $\theta=0$, $k=2$ and $F=0$. It suffices to choose $M\geq 3N$. By differentiating \eqref{problemaparabolico}, we find that $\p_t^p u$, $p\ge 1$, satisfies 
\begin{equation*}
\begin{cases}
\p_t^{p+1} u+\mathcal L\p_t^p u=F_p,\ &\text{in}\ \Omega\times (0,1],\\
\p_t^p u=D\p_t^p u=\cdots=D^{m-1}\p_t^p u=0 ,\ &\text{on}\ \partial\Omega\times (0,1],
\end{cases}
\end{equation*}
with
\begin{equation*}
F_p=(-1)^{m+1}\sum_{|\alpha|\leq 2m} \sum_{q=0}^{p-1} \binom{p}{q}\p_t^{p-q}a_{\alpha} \p_t^q \p_x^\alpha u.
\end{equation*}
Assume that \eqref{time_analy} holds up to $p-1$ for some $p\ge 1$ and apply the weighted $L^2$ estimate in Lemma \ref{Lema:regularidad1} with $\theta=0$ and $k=2$ to $\p_t^p u$ to obtain
\begin{equation*}
\|t^{p+1}\p_t^{p+1} u \|+\sum_{l=0}^{2m}\|t^{p+\frac l{2m}}D^{l} \p_t^p u \|
\leq N \left[2(p+1)\|t^p \p_t^{p} u \|+ \|t^{p+1}  F_p \| \right]\triangleq I_1+I_2.
\end{equation*}
By the induction,
\begin{equation*}
\|t^p \p_t^{p} u \| \leq M  \rho^{-p+1}p! \|u\|\, .
\end{equation*}
From \eqref{coeff:time_analytic} and induction
\begin{equation*}
\begin{split}
\|t^{p+1} F_p \| &\leq \sum_{|\alpha|\leq 2m} \sum_{q<p}\binom{p}{q}\varrho^{-1-p+q} (p-q)! \|t^{q+\frac{|\alpha|}{2m}}  \p_t^q \partial_x^\alpha  u\| \\
& \leq  \sum_{|\alpha|\leq 2m} \sum_{q<p}\binom{p}{q}\varrho^{-1-p+q} (p-q)! M \rho^{-q} (q+1)!\|u\|\\
& \leq N M p \sum_{q<p}\binom{p}{q} (p-q)!q!{\varrho}^{-p+q}\rho^{-q}\|u\| \\
& \leq M  \rho^{-p} (p+1)!\|u\|\frac{N\rho}{\varrho-\rho},
\end{split}
\end{equation*}
where the last inequality follows from Lemma \ref{L:postleibniz}. Adding $I_1$ and $I_2$, we get
\begin{equation*}
I_1+I_2 \leq M \rho^{-p}(p+1)! \|u\|  N \left(\rho +\frac{\rho}{\varrho-\rho}\right)
\end{equation*}
and the induction for $p$ follows after choosing $\rho=\rho(\varrho,n,m)$ small.
\end{proof}

Lemma \ref{Lem:interior} yields an interior quantitative estimate of spatial analyticity.
\begin{lemma}\label{Lem:interior} Let $0<\theta\le 1$, $0<\frac{R}{2}<r<R\le 1$, $B_R\subset \Omega$ and  $\mathcal L$
satisfy \eqref{coeff:space_analytic} for $p=0$ over $B_R\times [0,1]$. Then, there are $M=M(\varrho, n, m)$
and $\rho=\rho(\varrho,n,m)$, $0<\rho\le 1$, such that for all $\gamma\in \N^n$, the inequality
\begin{multline}\label{tan_analy}
(R-r)^{2m}\|t\peso \p_t\p_{x}^\gamma u\|_r+\sum_{k=0}^{2m}(R-r)^k\|t^{\frac{k}{2m}}\peso D^k\p_{x}^\gamma u\|_{r}\\
\leq M\left[\rho \theta^b (R-r) \right]^{-|\gamma|}|\gamma|!\|u\|_R
\end{multline}
holds when $u$ in $C^\infty(B_R\times [0,1])$ satisfies $\partial_t u+\mathcal Lu=0$ in $B_R\times [0,1]$.
\end{lemma}
\begin{proof} We prove \eqref{tan_analy} by induction on $|\gamma|$. When $|\gamma|=0$, by the weighted $L^2$ estimate in Lemma \ref{Lema:regularidad4} with $k=2$, $p=0$, $\delta=\frac{R-r}{2}$ and $F=0$, we have
\begin{multline*}
\|t\peso \p_t u \|_r+\|t\peso D^{2m} u \|_r\\
  \leq N \left[(R-r)^{-2m}\|t \peso u\|_{r+\delta} + \| \pesodos u \|_{r+\delta}\right]\\
  \leq N(R-r)^{-2m} \|u\|_R.
\end{multline*}
and Lemma \ref{Lema:interpolation} with $M= 2N$ implies
\begin{equation}\label{E: primercaso}
(R-r)^{2m}\|t\peso \p_t u \|_r+\sum_{l=0}^{2m}(R-r)^l\|t^{\frac l{2m}}\peso D^{l} u \|_r
\leq M  \|u\|_R,
\end{equation}
\color{black}

Next, assume that \eqref{tan_analy} holds for multi-indices $\gamma$, with $|\gamma|\leq l$, $l\ge 0$, and we  show  that \eqref{tan_analy} holds for any multi-index of the same form with $|\gamma|=l+1$. Differentiating  \eqref{problemaparabolicoboundary} we find that $\p_{x}^\gamma u$ satisfies
\begin{equation*}
\p_t\p_{x}^\gamma u+\mathcal L\p_{x}^\gamma u=F_\gamma,\ \text{in}\ B_R\times(0,1],
\end{equation*}
with
\begin{equation}\label{nonhom:tan}
F_\gamma=(-1)^{m+1}\sum_{|\alpha|\leq 2m} \sum_{\beta<\gamma} \binom{\gamma}{\beta} \p_{x}^{\gamma-\beta}a_\alpha \p_{x}^{\beta}\p_x^{\alpha}u.
\end{equation}
Applying the weighted $L^2$ estimate in Lemma \ref{Lema:regularidad4} to $\p_{x}^\gamma u$ with $p=0$, we get 
\begin{multline}\label{formulote:tan}
\|t\peso \p_t \p_{x}^\gamma u \|_r+\|t\peso D^{2m} \p_{x}^\gamma u \|_r 
\\ \leq N\left[ k \| e^{{-\frac{k-1}{k}}\theta t^{-\sigma}}\p_{x}^\gamma u \|_{r+\delta}\right.+\delta^{-2m}\|t \peso \p_{x}^\gamma u\|_{r+\delta}\\\left. +\|t\peso F_\gamma \|_{r+\delta}   \right]\triangleq I_1+I_2+I_3.
\end{multline}
{\it Estimate for $I_1$}: when $1\le |\gamma|\leq 2m$, choose $k=2$ and $\delta=(R-r)/2$ in \eqref{formulote:tan}. Also observe the bound 
\begin{equation}\label{unadelasclaves}
t^{-\alpha}e^{-\theta t^{-\beta}}\leq e^{-\frac{\alpha}{\beta}}\theta^{-\frac{\alpha}{\beta}}\left(\frac{\alpha}{\beta}\right)^{\frac{\alpha}{\beta}},\ \text{when}\ \alpha,\, \beta,\,\theta\ \text{and}\ t>0,
\end{equation}
which yields $$t^{-\frac{|\gamma|}{2m}}e^{-\frac{\theta}{4}t^{-\sigma}}\leq N \theta^{-b|\gamma|},\text{ when }|\gamma|\leq 2m\text{ and }t>0.$$
Thus, we get \color{black} 
\begin{equation}\label{112}
\begin{split}
\|\pesodos \p_{x}^\gamma u \|_{r+\delta}&= \|t^{-\frac{|\gamma|}{2m}}e^{-\frac{\theta}{4}t^{-\sigma}}t^{\frac{|\gamma|}{2m}}e^{-\frac{\theta}{4}t^{-\sigma}}  \p_{x}^\gamma u \|_{r+\delta}\\
&  \leq N\theta^{-b|\gamma|}\|t^\frac{|\gamma|}{2m} e^{-\frac{\theta}{4}t^{-\sigma}} D^{|\gamma|} u \|_{r+\delta},
\end{split}
\end{equation}
when $|\gamma|\leq 2m$.
From  \eqref{E: primercaso} 
$$\|t^\frac{|\gamma|}{2m} e^{-\frac{\theta}{4}t^{-\sigma}} D^{|\gamma|} u \|_{r+\delta}\leq M(R-r)^{-|\gamma|} \|u\|_R,$$
this, together with \eqref{112} shows that
\begin{equation*}
\|\pesodos \p_{x}^\gamma u \|_{r+\delta}\leq N M \left[\theta^b(R-r)\right]^{-|\gamma|} \|u\|_R\leq M \left[\rho\theta^b(R-r)\right]^{-|\gamma|} \|u\|_RN\rho.
\end{equation*}

If $|\gamma|>2m$, choose $k=|\gamma|$, $\delta=(R-r)/|\gamma|$ in \eqref{formulote:tan} and observe that there is a multi-index $\xi$, with $2m+|\xi|=|\gamma|$ and $|\p_{x}^{\gamma}u|\leq |D^{2m}\p_{x}^{\xi}u|$. 
Hence, from \eqref{unadelasclaves}
\begin{multline} \label{unaayuda}
\| e^{-\frac{|\gamma|-1}{|\gamma|}\theta t^{-\sigma}} \p_{x}^\gamma u \|_{r+\delta}
=\|t^{-1} e^{{-\frac{|\gamma|-1}{|\gamma|^2}}\theta t^{-\sigma}} t e^{{-\left(1-\frac{1}{|\gamma|}\right)^2}  \theta t^{-\sigma}}\p_{x}^\gamma u \|_{r+\delta}\\
\leq N \theta^{-(2m-1)} |\gamma|^{2m-1} \|t e^{-\left(1-\frac{1}{|\gamma|}\right)^2 \theta t^{-\sigma}}   D^{2m}\p_{x}^\xi u \|_{r+\delta}.
\end{multline}
By induction and because $R-r-\delta=\frac{|\gamma|-1}{|\gamma|}(R-r)$,
\begin{equation}\label{otraayuda}
\begin{split}
(R-r)^{2m}&\|t e^{-\left(1-\frac{1}{|\gamma|}\right)^2 \theta t^{-\sigma}}   D^{2m}\p_{x}^\xi u \|_{r+\delta}\\
&\leq  M \left[\rho \left(1-\frac{1}{|\gamma|} \right)^{2b} \theta^b \left(R-r-\delta \right) \right]^{-|\gamma|+2m} \\
&\times (|\gamma|-2m)!\|u\|\\
&=M \left(1-\frac{1}{|\gamma|} \right)^{-(2b+1)(|\gamma|-2m)}\left[\rho \theta^b (R-r) \right]^{-|\gamma|+2m}\\
&\times(|\gamma|-2m)!\|u\|_R\\
&\leq MN \left[\rho \theta^b (R-r) \right]^{-|\gamma|+2m}(|\gamma|-2m)!\|u\|_R,
\end{split}
\end{equation}
where the last inequality is a consequence of the estimate 
\begin{equation} \label{amiga}
\left(1-\frac{1}{|\gamma|} \right)^{-(2b+1)(|\gamma|-2m)}\leq N,\ \text{for all}\  \gamma \in \N^n.
\end{equation}
Plugging \eqref{otraayuda} into \eqref{unaayuda} and using that $|\gamma|^{2m}(|\gamma|-2m)!\leq N|\gamma|!$, we get
\begin{align*}
I_1&\leq N |\gamma| \| e^{-\frac{|\gamma|-1}{|\gamma|}\theta t^{-\sigma}} \p_{x}^\gamma u \|_{r+\delta}\\
&\leq M  \left[\rho \theta^b (R-r)\right]^{-|\gamma|} |\gamma|^{2m}(|\gamma|-2m)! \|u\|_R N\rho^{2m}\\
&\leq M  \left[\rho \theta^b (R-r)\right]^{-|\gamma|} |\gamma|! \|u\|_R (R-r)^{-2m}N\rho.
\end{align*}
{\it Estimate for $I_2$}: when $|\gamma|\leq 2m$, the term can be handled like the term $I_1$ in the case $|\gamma|\leq 2m$, but now one does not need to push inside $I_1$ the factor $t^{|\gamma|/2m}$ as we did in \eqref{112}. Here, from \eqref{E: primercaso} we get
\begin{equation*}
I_2\le M \left[\rho\theta^b(R-r)\right]^{-|\gamma|} \|u\|_R(R-r)^{-2m}N\rho.
\end{equation*}
When $|\gamma|>2m$, again $|\p_{x}^\gamma u|\leq |D^{2m}\p_{x}^{\xi}u|$, for some $\xi$ such that $2m+|\xi|=|\gamma|$.
By induction (recall that $\delta=(R-r)/|\gamma|$ was already chosen in the estimate for $I_1$, when $|\gamma|>2m$) we get 
\begin{align*}
I_2&\leq N(R-r)^{-2m}|\gamma|^{2m}\|t \peso\p_{x}^{\gamma}u\|_{r+\delta}\\
& \leq N(R-r)^{-2m}|\gamma|^{2m} \|t \peso D^{2m}\p_{x}^{\xi}u\|_{r+\delta} \\
&\leq N M \left[\rho \theta^b (R-r) \right]^{-|\gamma|+2m}|\gamma|^{2m}(|\gamma|-2m)! \|u\|_R(R-r)^{-4m}\\
& \leq M \left[\rho \theta^b (R-r) \right]^{-|\gamma|} |\gamma|! \|u\|_R(R-r)^{-2m} N\rho.
\end{align*}
{\it Estimate for $I_3$}: by the induction hypothesis and Lemma \ref{L:postleibniz}
\begin{equation*}
\begin{split}
&\|t\peso F_\gamma\|_{r+\delta}\leq N\sum_{|\alpha|\leq 2m} \sum_{\beta<\gamma} \binom{\gamma}{\beta}\varrho^{-|\gamma-\beta|}|\gamma-\beta|!\|t^{\frac{|\alpha|}{2m}} \peso D^{|\alpha|}\p_{x}^{\beta}u\|_{r+\delta}\\
&\leq NM \sum_{\beta < \gamma}\binom{\gamma}{\beta} \varrho^{-1-|\gamma-\beta|}|\gamma-\beta|! \left[\rho \theta^b (R-r) \right]^{-|\beta|}|\beta|!\|u\|_R (R-r)^{-2m} \\
& \leq NM \left[\theta^b (R-r) \right]^{-|\gamma|}\|u\|_R (R-r)^{-2m}\sum_{\beta <\gamma} \binom{\gamma}{\beta}|\gamma-\beta|!|\beta|! \varrho^{-|\gamma-\beta|}\rho^{-|\beta|}\\
&\leq  M \left[\rho\theta^b(R-r) \right]^{-|\gamma|}|\gamma|!\|u\|_R (R-r)^{-2m} \frac{N\rho}{\varrho-\rho}.
\end{split}
\end{equation*}
The bounds for $I_1$, $I_2$ and $I_3$ imply that
\begin{equation}\label{E: casi alli estas}  
I_1+I_2+I_3\le M  \left[\rho\theta^b(R-r) \right]^{-|\gamma|}|\gamma|!\|u\|_R (R-r)^{-2m} N \rho\left(1+\frac{1}{\varrho-\rho}\right).
\end{equation}\label{E: segundacosad}

We can write, $\gamma=\xi+e_i$, for some $\xi\in\N^n$ and $i=1,\dots,n$, and from the induction and \eqref{unadelasclaves}
\begin{equation}\label{E: segundocaso}
\|e^{-\theta t^{-\s}}\partial^\gamma_{x}u\|_r\le N\theta^{-b}\|t^{\frac 1{2m}}e^{-\frac\theta2 t^{-\s}}D\partial^{\gamma-e_i}_{x}u\|_r\le M \left[\rho\theta^b(R-r) \right]^{-|\gamma|}|\gamma|!\|u\|_RN\rho.
\end{equation}
Finally, Lemma \ref{Lema:interpolation}, \eqref{formulote:tan}, \eqref{E: casi alli estas} and \eqref{E: segundocaso} imply the desired result when $\rho=\rho(\varrho,n,m)$ is small.
\end{proof}
\begin{remark}\label{R: unacosiat} Lemma \ref{Lem:interior} also holds when the coefficients of $\mathcal L$ are measurable in the time variable and satisfy \eqref{coeff:space_analytic} for $p=0$ over $B_R\times [0,1]$. This follows from the interior $W^{2m,1}_2$ Schauder estimate in \cite[Theorem 2]{DongKim} and the weighted $L^2$ estimate in Lemma \ref{Lema:regularidad4}.
\end{remark}
Next we state the quantitative estimates of spatial analyticity in directions locally  tangent to the boundary of $\Omega$ that the methods in Lemma \ref{Lem:interior} yield. For this purpose, we flatten locally $B_R(q_0)\cap\p \Omega$, with $q_0\in \p \Omega$, by means of the analytic change of variables
\begin{align*}
y_n=x_n-\varphi(x'),\ \ \ & y_j=x_j,\ j=1\ldots,n-1,
\end{align*}
where $\varphi$ is the analytic function introduced in \eqref{E:descripcionfrontera}. The local change of variables does not modify the local conditions satisfied by $\mathcal L$ and without loss of generality we may assume that a solution to \eqref{problemaparabolico} verifies
\begin{equation}\label{problemaparabolicoboundary}
\begin{cases}
\p_tu+\mathcal Lu=0,\ &\text{in}\ B_R^+\times (0,1],\\
u=Du=\ldots=D^{m-1}u=0,\ \ &\text{in}\ \{x_n=0\}\cap\partial B_R^+ \times (0,1],
\end{cases}
\end{equation}
with $u$ in $C^\infty(B_R^+\times [0,1])$ and $0<R\le 1$.

Here, we use multi-indices of the form $\left(\gamma_1,\ldots,\gamma_{n-1},0\right)\in \N^n$ and write $\p_{x'}^\gamma$ instead of $\p_{x}^\gamma$ to emphasize that $\p_{x'}^{\gamma}$ does not involve derivatives with respect to the  variable $x_n$. Lemma \ref{Th:tan_analy} is proved as Lemma \ref{Lem:interior} but with Lemma \ref{Lema:regularidad4} replaced by Lemma \ref{Lema:regularidad3}. We omit the proof.

\begin{lemma}\label{Th:tan_analy} Let $0<\theta\le 1$, $0<\frac{R}{2}<r<R\le 1$ and  assume that $\mathcal L$ satisfies \eqref{coeff:space_analytic} for $p=0$ over $B_R^+\times[0,1]$. Then, there are $M=M(\varrho,n,m)$ and $\rho=\rho(\varrho,n,m)$, $0<\rho\le 1$,  such that for all $\gamma\in \N^n$ with $\gamma_n=0$, the inequality
\begin{multline*}
(R-r)^{2m}\| t\peso \p_t\p_{x'}^\gamma u \|_r'+\sum_{k=0}^{2m} (R-r)^k \|t^\frac{k}{2m} \peso D^k \p_{x'}^\gamma u \|_r' \\
\leq M  \left[\rho \theta^b (R-r)\right]^{-|\gamma|} |\gamma|!\|u\|_R',
\end{multline*}
holds when $u$ in $C^\infty(B_R^+\times [0,1])$ satisfies \eqref{problemaparabolicoboundary}.
\end{lemma}
\begin{remark}\label{R:5} Lemma \ref{Th:tan_analy} also holds when the coefficients of $\mathcal L$ are measurable in the time variable and satisfy \eqref{coeff:space_analytic} for $p=0$ over $B_R^+\times[0,1]$. It follows from the weighted $L^2$ estimate in Lemma \ref{Lema:regularidad3} and \cite[Theorem 4]{DongKim}.
\end{remark}

Next, combining Lemmas \ref{Th:time_analyticity} and \ref{Th:tan_analy} one can prove the following.
\begin{lemma}\label{Th:time_tan:analy}
Let $0<\theta\le 1$, $0<\frac{R}{2}<r<R\le 1$ and assume that $\mathcal L$ satisfies \eqref{coeff:space_analytic} and \eqref{coeff:time_analytic}. Then there are $M=M(\varrho, n, m)$ and $\rho=\rho(\varrho,n,m)$, $0<\rho\le 1$, such that for all $\gamma\in \N^n$, $\gamma_n=0$, and $p\in \N$, the inequality 
\begin{multline*}
(R-r)^{2m} \|t^{p+1} \peso \p_t^{p+1} \p_{x'}^\gamma u \|_r' + \sum_{k=0}^{2m} (R-r)^k \|t^{p+\frac{k}{2m}} \peso D^k \p_t^{p}  \p_{x'}^\gamma  u \|_r' \\
\leq M \rho^{-p} \left[\rho  \theta^b (R-r) \right]^{-|\gamma|} (p+|\gamma|+1)! \|u\| 
\end{multline*}
holds when $u$ is a solution to \eqref{problemaparabolico} and \eqref{problemaparabolicoboundary}.
\end{lemma}
\begin{proof}
We proceed by induction on $p$ and within each $p$-case we proceed by induction on $|\gamma|$. The case $p=0$ and $\gamma\in \N^n$ with $\gamma_n=0$ follows from Lemma \ref{Th:tan_analy}, whereas the case $|\gamma|=0$ with arbitrary $p\ge 0$ follows from Lemma \ref{Th:time_analyticity}. 
Thus, we may in what follows assume that $|\gamma|\ge 1$. By differentiation of \eqref{problemaparabolicoboundary}, $\partial_t^p\partial^\gamma_{x'}u$ satisfies
\begin{equation*}
\begin{cases}
\p_t^{p+1} \p_{x'}^\gamma u+\mathcal L\p_t^{p} \p_{x'}^\gamma u=F_{(\gamma,p)},\ &\text{in}\ B_R^+\times (0,T],\\
\p_t^p \p_{x'}^\gamma u=D\p_t^p \p_{x'}^\gamma u=\ldots=D^{m-1}\p_t^p \p_{x'}^\gamma u=0 ,\ &\text{on}\ \{x_n=0\}\cap\partial B_R^+\times (0,T],
\end{cases}
\end{equation*}
with
\begin{equation}\label{nonhom:time_tan}
F_{(\gamma,p)}=(-1)^{m+1}\sum_{|\alpha|\leq 2m} \sum_{\substack{(q,\beta)\\ <(p,\gamma)}} \binom{p}{q} \binom{\gamma}{\beta}  \p_t^{p-q}\p_{x'}^{\gamma-\beta}a_\alpha  \p_t^q\p_{x'}^{\beta}\p_x^{\alpha}u.
\end{equation}
By the weighted $L^2$ estimate in Lemma \ref{Lema:regularidad3} applied to $\partial_t^p\partial^\gamma_{x'}u$,
\begin{multline}\label{generica}
\|t^{p+1} \peso \p_t^{p+1} \p_{x'}^{\gamma} u\|_r'+ \|t^{p+1} \peso D^{2m} \p_t^p \p_{x'}^{\gamma}u\|_r' \\
\leq N\left[ (p+k) \|t^p e^{-\theta \frac{k-1}{k}t^{-\sigma}} \p_t^p \p_{x'}^{\gamma}u\|_{r+\delta}'+ 
\delta^{-2m} \|t^{p+1} \peso  \p_t^p \p_{x'}^{\gamma} u\|_{r+\delta}'\right.\\ \left.
+ \|t^{p+1}\peso F_{(\gamma,p)}\|_{r+\delta}'\right]\triangleq I_1+I_2+I_3.
\end{multline}
\textit{Estimate for $I_1$}: if $|\gamma|\leq 2m$, take $k=2$ and $\delta=(R-r)/2$ in \eqref{generica}. Taking into account that $(p+1)!\leq N (p+|\gamma|)!$, \eqref{unadelasclaves} and Lemma \ref{Th:time_analyticity}, we obtain
\begin{equation*}
\begin{split}
I_1 &\leq N(p+2) \|t^p \pesodos \p_t^p\p_{x'}^\gamma u\|_{r+\delta}'\\
&\leq  N(p+2) \theta^{-b|\gamma|} \|t^{p+\frac{|\gamma|}{2m}}e^{-\frac{\theta}{4}t^{-\sigma}}D^{|\gamma|} \p_t^p u\|_{r+\delta}'\\
& \leq M \rho^{-p}\left[\rho\theta^b (R-r)\right]^{-|\gamma|} (p+|\gamma|)!\|u\|N\rho(R-r)^{-2m}.
\end{split}
\end{equation*}
In the previous chain of inequalities we used that
$$\|t^{p+\frac{|\gamma|}{2m}}e^{-\frac{\theta}{4}t^{-\sigma}}D^{|\gamma|} \p_t^p u\|_{r+\delta}'\leq M \|t^{p+\frac{|\gamma|}{2m}}e^{-\frac{\theta}{4}t^{-\sigma}}D^{|\gamma|} \p_t^p u\|$$
and applied Lemma \ref{Th:time_analyticity}.

If $|\gamma|>2m$, choose $k=|\gamma|$ and $\delta=(R-r)/|\gamma|$ in \eqref{generica}. There is a multi-index $\xi\in \N^n$ with $\xi_n=0$ such that $2m+|\xi|=|\gamma|$ and $|\p_t^p \p_{x'}^{\gamma}u|\leq |D^{2m} \p_t^p \p_{x'}^{\xi}u|$ and from \eqref{unadelasclaves}
\begin{equation}\label{E: laprimerisima2}
\begin{split}
I_1 & \leq N(p+|\gamma|)\|t^p e^{-\theta\frac{|\gamma|-1}{|\gamma|} t^{-\sigma}} \p_t^p\p_{x'}^\gamma u \|_{r+\delta}'\\
&=N(p+|\gamma|) \|t^{-1}e^{-\theta\frac{|\gamma|-1}{|\gamma|^2} t^{-\sigma}} t^{p+1} e^{-\theta\left(1-\frac{1}{|\gamma|}\right)^2 t^{-\sigma}} \p_t^p\p_{x'}^\gamma u \|_{r+\delta}'\\
& \leq N (p+|\gamma|)|\gamma|^{2m-1} \theta^{-(2m-1)} \|t^{p+1} e^{-\theta\left(1-\frac{1}{|\gamma|}\right)^2 t^{-\sigma}}D^{2m}\p_t^p\p_{x'}^{\xi} u \|_{r+\delta}'.
\end{split}
\end{equation}
We apply the induction hypothesis and proceed as in \eqref{otraayuda} using \eqref{amiga} to get that 
\begin{multline}\label{Elasrundisima}
\|t^{p+1} e^{-\theta\left(1-\frac{1}{|\gamma|}\right)^2 t^{-\sigma}} D^{2m}\p_t^p\p_{x'}^{\xi} u \|_{r+\delta}'\\
\leq N M \rho^{-p} \left[\rho \theta^b \left(R-r\right) \right]^{-|\gamma|+2m} (p+|\gamma|-2m+1)!\|u\|(R-r)^{-2m}.
\end{multline} 
From
\begin{equation*}
|\gamma|^{2m-1}(p+|\gamma|-2m+1)!(p+|\gamma|)\leq N (p+|\gamma|+1)!,
\end{equation*}
\eqref{E: laprimerisima2} and \eqref{Elasrundisima}
\begin{equation*}
I_1\le M \rho^{-p} \left[\rho \theta^b (R-r) \right]^{-|\gamma|}  (p+|\gamma|+1)!\|u\|N\rho(R-r)^{-2m}.
\end{equation*}
{\it Estimate for $I_2$}: For $|\gamma|\leq 2m$, we set $\delta=(R-r)/2$ and because $\theta$ and $R\le1$, Lemma \ref{Th:time_analyticity} shows that
\begin{equation*}
\begin{split}
I_2&\leq N(R-r)^{-2m} \|t^{p+\frac{|\gamma|}{2m}}\peso D^{|\gamma|} \p_t^p u\|_{r+\delta}' \\
&\leq N(R-r)^{-2m} M \rho^{-p} (p+1)! \|u\| \\
&\leq M \left[\rho\theta^b (R-r)\right]^{-|\gamma|} \rho^{-p} (|\gamma|+p+1)! \|u\|N\rho(R-r)^{-2m}.
\end{split}
\end{equation*}
If $|\gamma|> 2m$, we have already chosen $\delta=(R-r)/|\gamma|$ and there is $\xi\in \N^n$, with $\xi_n=0$, $2m+|\xi|=|\gamma|$ and  $|\p_t^p \p_{x'}^{\gamma}u|\leq |D^{2m}\p_t^p \p_{x'}^{\xi}u|$. By the induction hypothesis and taking into account that
\begin{equation*}
|\gamma|^{2m}(p+|\gamma|-2m+1)!\leq N (p+|\gamma|+1)!,
\end{equation*}
we get
\begin{equation*}
\begin{split}
I_2&\leq N (R-r)^{-2m} |\gamma|^{2m}\|t^{p+1} \peso D^{2m}\p_{x'}^{\xi} \p_t^p u\|_{r+\delta}' \\
&\leq N (R-r)^{-2m}|\gamma|^{2m} M \rho^{-p} \left[\theta^b \rho (R-r) \right]^{-(|\gamma|-2m)} (p+|\gamma|-2m+1)! \|u\| \\
&\leq  M \rho^{-p} \left[ \rho\theta^b(R-r)\right]^{-|\gamma|}  (p+|\gamma|+1)! \|u\|  N\rho(R-r)^{-2m}.
\end{split}
\end{equation*}
{\it Estimate for $I_3$}: by the induction hypothesis on multi-indices $(q,\beta)<(p,\gamma)$ and Lemma \ref{L:postleibniz} for $\N^{n+1}$,
\begin{equation*} 
\begin{split}
I_3&=\|t^{p+1} \peso F_{(\gamma,p)}\|_{r+\delta}' \\ 
&\leq N\sum_{|\alpha|\leq 2m} \sum_{\substack{(q,\beta)\\ <(p,\gamma)}} \binom{p}{q} \binom{\gamma}{\beta} \varrho^{-p+q -|\gamma|+|\beta|}(p-q+|\gamma|-|\beta|)!\\
&\times \|t^{p+\frac{|\alpha|}{2m}}\peso D^{|\alpha|} \p_t^q \p_{x'}^\beta u\|_{r+\delta}'\\
& \leq  NM  \left[\theta^b (R-r)\right]^{-|\gamma|} (p+|\gamma|)\|u\| (R-r)^{-2m}\\
& \times \sum_{\substack{(q,\beta)\\ <(p,\gamma)}} \binom{p}{q}\binom{\gamma}{\beta} (p-q+|\gamma-\beta|)! (q+|\beta|)! \varrho^{-p+q-|\gamma-\beta|}  \rho^{-q-|\beta|}   \\
&\leq  M   \rho^{-p}\left[\rho \theta^b (R-r)\right]^{-|\gamma|} (p+|\gamma|+1)!\|u\|(R-r)^{-2m}\frac{N\rho}{\varrho-\rho}.
\end{split}
\end{equation*}
Thus,
\begin{equation}\label{E: alli}
I_1+I_2+I_3 \leq M \rho^{-p}\left[\rho\theta^b (R-r)\right]^{-|\gamma|} (p+|\gamma|)!\|u\|N\rho(R-r)^{-2m},
\end{equation}
and Lemma \ref{Th:time_tan:analy} follows from \eqref{generica}, \eqref{E: alli}, Lemma \ref{Lema:interpolation} and the induction hypothesis for $(p-1,\gamma)$, when $\rho=\rho(\varrho,n,m)$ is small.
\end{proof}

Finally, Theorem \ref{lagrananaliticidad} follows from the embedding \cite{Friedman1}
\begin{equation*}
\|\varphi\|_{L^\infty(\Omega)}\le C(n) \sum_{|\alpha|\le \left[\frac{n}{2}\right]+1}\|D^{|\alpha|}\varphi\|_{L^2(\Omega)},\ \text{for}\ \varphi\in C^\infty(\overline\Omega),
\end{equation*}
the inequality
\begin{equation*}
\|f\|_{L^\infty(I)}\le |I|^{\frac12}\|f'\|_{L^2(I)}+|I|^{-\frac 12}\|f\|_{L^2(I)},\ \text{for}\ f\in C^1(I),
\end{equation*}
\color{black}
with $I$ an interval in $\R$ and Lemma \ref{Lem:normal_analy}.
\begin{lemma}\label{Lem:normal_analy} Let $0 <\theta\le 1$, $0<\frac{R}{2}<r<R\le1$ and $\mathcal L$ satisfy \eqref{coeff:space_analytic} and \eqref{coeff:time_analytic}. Then, there are $M=M(\varrho,n,m)$ and $\rho=\rho(\varrho,n,m)$, $0<\rho\le 1$, such that
\begin{multline}\label{normal_analy}
\|t^{p}\peso \p_t^{p} \p_n^l \p_{x'}^\gamma u\|_{r}'\\\leq M \rho^{-p-l}\left[\rho\,\theta^b  (R-r)\right]^{-l-|\gamma|}(p+l+|\gamma|+1)!\,\|u\|
\end{multline}
holds when $u$ is a solution to \eqref{problemaparabolico} and \eqref{problemaparabolicoboundary}. Here, $\p_n$ denotes differentiation with respect to the variable $x_n$.
\end{lemma}
\begin{proof} A solution to \eqref{problemaparabolicoboundary} satisfies
\begin{equation}\label{eq:diff:normal}
\p_t^{p+1}\p_n^l\p_{x'}^{\gamma}  u+\mathcal L\p_t^{p} \p_n^l \p_{x'}^{\gamma}  u=F_{(p,l,\gamma)},\ \text{in}\ B_R^+\times(0,1],
\end{equation}
with
\begin{equation*}
F_{(p,l,\gamma)}=(-1)^{m+1}\sum_{|\alpha|\leq 2m} \sum_{\substack{(q,j,\beta)\\<(p,l,\gamma)}} \binom{p}{q} \binom{l}{j} \binom{\gamma}{\beta} \p_t^{p-q} \p_n^{l-j} \p_{x'}^{\gamma-\beta} a_{\alpha} \p_t^{q} \p_n^{j} \p_{x'}^{\beta}  \p_x^\alpha u.
\end{equation*}

Because of \eqref{parabolicidad}, $a_{2m e_n}\geq \varrho>0$ in $\Omega\times [0,1]$, and one can solve for $\p_t^{p} \p_n^{l+2m} \p_{x'}^{\gamma}u$ in \eqref{eq:diff:normal}. Substituting into that formula $l$ by $l-2m+1$, when $l\geq 2m$, we have
\begin{multline}\label{despeje}
|\p_t^{p} \p_n^{l+1}\p_{x'}^{\gamma}u|\leq \frac{1}{|a_{2me_n}|}\left[|\p_{t}^{p+1} \p_n^{l-2m+1} \p_{x'}^\gamma    u|+|F_{(p,l-2m+1,\gamma)}|\right]\\+\frac{1}{|a_{2me_n}|}\sum_{\substack{|\alpha|\leq 2m\\ 
\alpha_n\leq 2m-1}} \left\|a_{\alpha} \right\|_{L^\infty(\Omega\times (0,1))}|\p_t^p\p_n^{l-2m+1} \p_{x'}^{\gamma}  \p_{x}^{\alpha}u|.
\end{multline}
We prove \eqref{normal_analy} by induction on the quantity $2mp+l+|\gamma|$ with $M$ the same constant as in Lemma \ref{Th:time_tan:analy}. If $2mp+l+|\gamma|\leq 2m$, then $l\leq 2m$ and \eqref{unadelasclaves} and Lemma  \ref{Th:time_tan:analy} show that
\begin{equation*}
\begin{split}
&\|t^{p}\peso \p_t^{p} \p_n^l\p_{x'}^\gamma u\|_{r}'\le\|t^{-\frac l{2m}}e^{-\frac{\theta}{1+|\gamma|}t^{-\sigma}} t^{p+\frac{l}{2m}} e^{-\frac{\theta|\gamma|}{1+|\gamma|}t^{-\sigma}} D^l\p_t^{p}\p_{x'}^\gamma u\|_{r}'\\
&\le N\theta^{-lb}\left(1+|\gamma|\right)^{\frac{(2m-1)l}{2m}}\|t^{p+\frac{l}{2m}} e^{-\frac{\theta|\gamma|}{1+|\gamma|}t^{-\sigma}} D^l\p_t^{p}\p_{x'}^\gamma u\|_{r}'\\
&\leq N\theta^{-lb}\left(1+|\gamma|\right)^l \left(R-r\right)^{-l}M \rho^{-p}\left[\rho\theta^b  (R-r)\right]^{-|\gamma|}(p+|\gamma|+1)!\,\|u\|\\
&\leq M \rho^{-p-l}\left[\rho\theta^b  (R-r)\right]^{-l-|\gamma|}(p+l+|\gamma|+1)!\,\|u\|N\rho^{2l},
\end{split}
\end{equation*}
where the last inequality holds because
\begin{equation*}
\left(1+|\gamma|\right)^{l}\left(p+|\gamma|+1\right)!\le N\left(p+l+|\gamma|+1\right)!\, .
\end{equation*}
Also, \eqref{normal_analy} holds when $l=0$ from Lemma \ref{Th:time_tan:analy}. Thus, \eqref{normal_analy} holds, when $2mp+l+|\gamma|\leq 2m$ and  $l\le 2m$, provided that $\rho$ is small.  

Assume now that \eqref{normal_analy} holds when $2mp+l+|\gamma|\leq k$, for some fixed $k\ge 2m$ and we shall prove it holds for $2mp+l+|\gamma|=k+1$. 

In the same way as for the case $k=2m$, Lemma \ref{Th:time_tan:analy} shows that \eqref{normal_analy} holds, when $2mp+l+|\gamma|=k+1$ and $l\leq 2m$, provided that $\rho$ is small. So, let us now assume that \eqref{normal_analy} holds for $2mp+j+|\gamma|=k+1$ and $j=0,\ldots,l$, for some $l\geq 2m$ and prove that it holds for $2mp+j+|\gamma|=k+1$ with $j=l+1$. Let then $\gamma$ and $p$ be such that $2mp+\left(l+1\right)+|\gamma|=k+1$. From \eqref{despeje} and because $a_{2me_n}\geq \varrho$,  we obtain
\begin{equation*}
\begin{split}
\|t^{p}\peso & \p_t^{p} \p_n^{l+1}\p_{x'}^{\gamma} u\|_r'\\
& \leq \varrho^{-1} \left[ \|t^{p} \peso \p_{t}^{p+1} \p_n^{l-2m+1} \p_{x'}^\gamma    u\|_r'+\|t^{p}\peso F_{(p,l-2m+1,\gamma)}\|_r'\right]\\
&+\varrho^{-1} \sum_{\substack{|\alpha|\leq 2m\\ \alpha_n\leq 2m-1}} \left\|a_{\alpha} \right\|_{L^\infty(Q)}\|t^{p}\peso \p_t^p \p_n^{l-2m+1} \p_{x'}^\gamma  \p_{x}^{\alpha}u\|_r'\\
&\triangleq H_1+H_2+H_3.
\end{split}
\end{equation*}
{\it Estimate for $H_1$}: the multi-indices involved in this term satisfy 
\begin{equation*}
2m(p+1)+l-2m+1+|\gamma|=k+1
\end{equation*}
and the total number of $x_n$ derivatives involved is less or equal than $l$. From the induction hypothesis and \eqref{unadelasclaves}, we can estimate $H_1$ as follows
\begin{equation*}
\begin{split}
\|t^{p}&\peso \p_t^{p+1} \p_n^{l-2m+1}  \p_{x'}^{\gamma}u \|_r' \\
&=\|t^{-1}e^{-\frac{\theta}{l+|\gamma|}t^{-\sigma}} t^{p+1}e^{-\theta\frac{l+|\gamma|-1}{l+|\gamma|}t^{-\sigma}} \p_t^{p+1} \p_n^{l-2m+1}  \p_{x'}^{\gamma}u \|_r'\\
&\leq N \theta^{-(2m-1)}(l+|\gamma|)^{2m-1}\|t^{p+1} e^{-\theta\frac{l+|\gamma|-1}{l+|\gamma|}t^{-\sigma}} \p_t^{p+1} \p_n^{l-2m+1} \p_{x'}^\gamma  u \|_r'\\
& \leq N \theta^{-(2m-1)} (l+|\gamma|)^{2m-1} M \rho^{-p-1-(l-2m+1)} \left[ \rho\theta^b (R-r)\right]^{-(l-2m+1)-|\gamma|} \\
& \times (p+l-2m+|\gamma|+3)!\, \|u\|\\
& \leq  M \rho^{-p-(l+1)} \left[\rho\theta^b (R-r)\right]^{-(l+1)-|\gamma|} \\
&\times (p+l+|\gamma|+2)!\, \|u\| N\rho^{4m-1},
\end{split}
\end{equation*}
where the last inequality holds because
\begin{equation*}
\left(l+|\gamma|\right)^{2m-1}\left(p+l-2m+|\gamma|+3\right)!\le N\left(p+l+|\gamma|+2\right)!\, ,
\end{equation*}
when $p+l+|\gamma|+2\ge 2m$. Thus, 
\begin{multline}\label{E: acotacion H1}
H_1\leq M \rho^{-p-(l+1)} \left[ \rho\theta^b (R-r)\right]^{-(l+1)-|\gamma|}\\
 \times (p+(l+1)+|\gamma|+1)!\, \|u\| N \rho^{4m-1}.
\end{multline}
{\it Estimate for $H_2$}:  we expand this term and obtain
\begin{equation} \label{6}
\begin{split}
H_2&\leq N\sum_{|\alpha|\leq2m} \sum_{\substack{(q,j,\beta)\\<(p,l-2m+1,\gamma)}} \binom{p}{q}\binom{l-2m+1}{j}\binom{\gamma}{\beta}\\
& \times \varrho^{-1-(p-q)-(l-2m+1-j)-|\gamma-\beta|} (p-q+l-2m+1-j+|\gamma-\beta|)!\\
&\times \|t^{q}\peso \p_t^{q} \p_n^j \p_{x'}^{\beta}  \p_{x}^{\alpha}u\|_r.
\end{split}
\end{equation}
The multi-indices involved in the derivatives of $u$ that appear in \eqref{6} satisfy $2mq+j+|\alpha|+|\beta|<2mp+l+1+|\gamma|=k+1$ and we already know how to control these derivatives by the first induction hypothesis. In fact, if we write $\alpha=(\alpha',\alpha_n)$ and because $\alpha_n$ is related to normal derivatives, we get
\begin{multline} \label{7}
\|t^{q}\peso \p_t^{q}\p_n^j \p_{x'}^{\beta}  \p_{x}^{\alpha}u\|_r' \\
\leq M \rho^{-q-j-\alpha_n}\left[\rho\theta^{b}  (R-r)\right]^{-j-|\beta|-|\alpha|}(q+j+|\beta|+|\alpha|+1)!\,\|u\|.
\end{multline}
The sum in \eqref{6}  runs over $\{(q,j,\beta)<(p,l-2m+1,\gamma)\}$ and $|\alpha|\leq 2m$ and inside the sum \eqref{6}, $j+\alpha_n+|\alpha|\le l+2m+1$, $j+|\beta|+|\alpha|\le l+1+|\gamma|$ and $q+j+|\beta|\leq p+l-2m+|\gamma|$. Also,
\begin{equation*}
\frac{\left(q+j+|\beta|+|\alpha|+1\right)!}{\left(q+j+|\beta|\right)!}\le \frac{\left(p+l+|\gamma|+1\right)!}{\left(p+l-2m+|\gamma|\right)!}\ .
\end{equation*}
These and \eqref{7} show that for all such $(q,j,\beta)$ and $\alpha$
\begin{multline} \label{71}
\|t^{q}\peso  \p_t^{q}\p_n^j \p_{x'}^{\beta} \p_{x}^{\alpha}u\|_r' 
\leq M \rho^{-l-2m-1}\left[\theta^b (R-r)\right]^{-l-1-|\gamma|} \rho^{-q-j-|\beta|}\\
\times \frac{(p+l+|\gamma|+1)!}{(p+l-2m+|\gamma|)!}(q+j+|\beta|)!\,\|u\|.
\end{multline}
Plugging \eqref{71} into \eqref{6} yields
\begin{equation}\label{72}
\begin{split}
H_2&\leq NM \rho^{-l-2m-1}\left[\theta^b(R-r)\right]^{-l-1-|\gamma|}\frac{(p+l+|\gamma|+1)!}{(p+l-2m+|\gamma|)!}\,  \|u\| \\
&\times  \sum_{\substack{(q,j,\beta)\\<(p,l-2m+1,\gamma)}} \binom{p}{q}\binom{l-2m+1}{j} \binom{\gamma}{\beta} \\
&\times (p-q+l-2m+1-j+|\gamma-\beta|)!(q+j+|\beta|)!\\
&\times \varrho^{-(p-q)-(l-2m+1-j)-|\gamma-\beta|}\rho^{-q-j-|\beta|}
\end{split}
\end{equation}
and Lemma \ref{L:postleibniz} shows that the above sum is bounded by
\begin{equation*}
\rho^{-p-l+2m-1-|\gamma|}(p+l-2m+1+|\gamma|)!\frac{\rho}{\varrho-\rho}\, .
\end{equation*}
The later and \eqref{72} imply
\begin{multline}\label{E: acotacion H2}
H_2
\leq M \rho^{-p-(l+1)}\left[\rho\theta^b(R-r)\right]^{-(l+1)-|\gamma|}(p+(l+1)+|\gamma|+1)!\,\|u\|\,\frac{N\rho}{\varrho-\rho}\, .
\end{multline}
{\it Estimate for $H_3$}: the multi-indices involved in the sum run over 
\begin{equation*}
\{\alpha: |\alpha|\le 2m: \alpha_n\le 2m-1\},
\end{equation*}
the multi-indices involved in the derivatives of $u$ which appear in $H_3$ satisfy 
\begin{equation*}
2mp+\left(l-2m+1+\alpha_n\right)+|\gamma|+|\alpha'|\le k+1,
\end{equation*} 
with a total number of $x_n$ derivatives equal to $\alpha_n+l-2m+1\leq l$, so we are within previous steps of the induction process and $0<\rho<1$. Accordingly, applying the second induction hypothesis one gets
\begin{equation}\label{E: acotacion H3}
H_3\leq  M  \rho^{-p-(l+1)}\left[\theta^b \rho (R-r) \right]^{-(l+1)-|\gamma|}\times (p+(l+1)+|\gamma|+1)!\,\|u\|N\rho.
\end{equation}
Now, \eqref{normal_analy} when $2mp+\left(l+1\right)+|\gamma|=k+1$, follows from \eqref{E: acotacion H1}, \eqref{E: acotacion H2} and \eqref{E: acotacion H3}, when $\rho=\rho(\varrho,n,m)$ is chosen small.
\end{proof}
\begin{remark}\label{R: alguno}
Choosing $\theta=t^\sigma$ in Lemma \ref {Lem:normal_analy}, one  recovers \eqref{E: kimalito}.
\end{remark}
Next we give a proof of the claim in the second paragraph in Remark \ref{R:1}. In Lemma \ref{Lem:tan_time:gevrey} we give details only for the interior case. Lemma \ref{Lem:tan_time:gevrey}  also holds near the boundary when the boundary is flat and for tangential derivatives $\gamma\in\N^n$ with $\gamma_n=0$. Then, as in Lemma \ref{Lem:normal_analy}, one can extend the result to all the derivatives by showing that there are $M=M(\varrho, n, m)$ and $\rho=\rho(\varrho,n,m)$, $0<\rho\le 1$, such that
\begin{equation*}
\|t\peso \p_t^{p} \p_n^l \p_{x'}^\gamma u\|_{r}'\\\leq M \rho^{-l} \left[\rho\theta^{b} (R-r)\right]^{-2mp-|\gamma|-l}(2mp+|\gamma|+l)!\|u\|_R'
\end{equation*}
when $u$ satisfies \eqref{problemaparabolicoboundary}, $\frac R2<r<R$ and (1.10) holds over $B_R(x_0)\cap\overline\Omega$.

\begin{lemma}\label{Lem:tan_time:gevrey} Let $0<\theta\le 1$, $0<\frac{R}{2}<r<R\le 1$ and $\mathcal L$ satisfy \eqref{coeff:space_analytic}. Then there are $M=M(\varrho, n, m)$ and $\rho=\rho(\varrho,n,m)$, 
$0<\rho\le 1$, such that for any $\gamma\in \N^n$ and $p\in \N$,
\begin{multline} \label{tan_time:gevrey}
(R-r)^{2m}\|t\peso \p_t^{p+1}\p_{x}^\gamma u\|_r+\sum_{k=0}^{2m}(R-r)^k\|t^{\frac{k}{2m}}\peso D^k \p_t^p\p_{x}^\gamma  u\|_{r}\\
\leq M \left[\rho \theta^{b} (R-r)\right]^{-2mp-|\gamma|} (2mp+|\gamma|)!\|u\|_R
\end{multline}
holds when $u$ in $C^\infty(B_R\times [0,1])$ satisfies $\p_t u+ \mathcal Lu=0$ in $B_R\times [0,1]$.
\end{lemma}
\begin{proof}
We prove \eqref{tan_time:gevrey} by induction on $p$ and then by induction on $|\gamma|$. When $p=0$, \eqref{tan_time:gevrey} is the estimate in Lemma \ref{Lem:interior}. Assume \eqref{tan_time:gevrey} holds up to $p-1$ for some $p\ge 1$. Then,
\begin{equation*}
\p_t^{p+1} \p_{x}^\gamma u+ \mathcal L\p_t^{p} \p_{x}^\gamma u=F_{\gamma,p},\ \text{in}\ B_R\times (0,1],\\
\end{equation*}
with
\begin{equation*}
F_{\gamma,p}=(-1)^{m+1}\sum_{|\alpha|\leq 2m} \sum_{\substack{(q,\beta)\\ <(p,\gamma)}} \binom{p}{q} \binom{\gamma}{\beta}  \p_t^{p-q}\p_{x}^{\gamma-\beta}a_\alpha  \p_t^q\p_{x}^{\beta}\p_x^{\alpha}u.
\end{equation*}
Apply the weighted $L^2$ estimate in Lemma \ref{Lema:regularidad4} with $p=0$, $k=p+|\gamma|+1$ and $\delta=(R-r)/(p+|\gamma|+1)$ to $\p_t^{p} \p_{x}^\gamma u$. It gives,
\begin{equation*} 
\begin{split}
&\|t \peso \p_t^{p+1} \p_x^\gamma u \|_r+\|t \peso D^{2m} \p_t^p \p_x^\gamma u \|_r \leq\\
& N \left[(|\gamma|+p)\| e^{-\theta \frac{|\gamma|+p}{|\gamma|+p+1}t^{-\sigma}} \p_t^{p} \p_x^\gamma u \|_{r+\delta}
+\frac{(|\gamma|+p+1)^{2m}}{(R-r)^{2m}}\|t e^{-\theta t^{-\sigma}} \p_t^{p} \p_x^\gamma u \|_{r+\delta}\right.\\ 
&\left.+\|t \peso F_{\gamma,p} \|_{r+\delta} \right]\triangleq I_1+I_2+I_3.
\end{split}
\end{equation*}
{\it Estimate for $I_1$}:  by induction hypothesis for $p-1$ and \eqref{unadelasclaves}
\begin{equation*}
\begin{split}
&\|e^{-\theta \frac{|\gamma|+p}{|\gamma|+p+1}t^{-\sigma}} \p_t^p \p_x^\gamma u \|_{r+\delta}
=\|t^{-1}e^{-\theta \frac{|\gamma|+p}{(|\gamma|+p+1)^2}t^{-\sigma}} t e^{-\theta \left(\frac{|\gamma|+p}{|\gamma|+p+1}\right)^2 t^{-\sigma}} \p_t^p \p_x^\gamma u \|_{r+\delta}\\
&\leq N  \theta^{-2mb}(|\gamma|+p)^{2m-1} \|t e^{-\theta \left(\frac{|\gamma|+p}{|\gamma|+p+1}\right)^2 t^{-\sigma}} \p_t^p \p_x^\gamma u \|_{r+\delta}\\
&\leq N \theta^{-2mb}(|\gamma|+p)^{2m-1} M \left[\rho\theta^b (R-r) \right]^{-2m(p-1)-|\gamma|}(2m(p-1)+|\gamma|)!\\
&\times \left(1+\frac{1}{|\gamma|+p}\right)^{(6m-2)(p+|\gamma|)}\|u\|_R(R-r)^{-2m}\\
&\leq M \left[\rho\theta^b (R-r) \right]^{-2mp-|\gamma|}(|\gamma|+p)^{2m-1}(2m(p-1)+|\gamma|)!\|u\|_R N\rho.
\end{split}
\end{equation*}
This and $(|\gamma|+p)^{2m}(2m(p-1)+|\gamma|)!\leq N(2mp+|\gamma|)!$, give
\begin{equation*}
I_1\leq M \left[\rho\theta^b (R-r) \right]^{-2mp-|\gamma|}(2mp+|\gamma|)!\|u\|_RN\rho(R-r)^{-2m}.
\end{equation*}
{\it Estimate for $I_2$}: by induction hypothesis for $p-1$
\begin{equation*}
\|t \peso \p_t^p \p_x^\gamma u \|_{r+\delta} \leq M \left[\rho\theta^b (R-r)\right]^{-2mp-|\gamma|}\left(2m(p-1)+|\gamma|\right)! \|u\|_R  N\rho
\end{equation*}
and
\begin{equation*}
I_2\leq M\left[\rho\theta^b (R-r)\right]^{-2mp-|\gamma|}\left(2mp+|\gamma|\right)!\|u\|_RN\rho(R-r)^{-2m}.
\end{equation*}
{\it Estimate for $I_3$}: by induction on $(q,\beta)<(p,\gamma)$ and Lemma \ref{L:postleibniz} for $\N^{n+1}$
\begin{equation*}
\begin{split}
&\|t\peso F_{\gamma,p}\|_{r+\delta}\\
&\leq N\sum_{|\alpha|\leq 2m} \sum_{\substack{(q,\beta)\\ <(p,\gamma)}} \binom{p}{q}\binom{\gamma}{\beta} \varrho^{-p+q-|\gamma-\beta|}(p-q+|\gamma|-|\beta|)!\\
&\times \|t^{|\alpha|}{2m}\peso D^{|\alpha|}\p_t^q \p_x^\beta  u\|_{r+\delta}\\
&\le NM\left[\theta^b(R-r)\right]^{-2mp-|\gamma|} \frac{(2mp+|\gamma|)!}{(p+|\gamma|)!}\rho^{-(2m-1)p}\|u\|_R(R-r)^{-2m}\\
&\times \sum_{\substack{(q,\beta)\\ <(p,\gamma)}}\binom pq\binom \gamma\beta(p-q+|\gamma|-|\beta|)!(q+|\beta|)!\varrho^{-p+q-|\gamma|+|\beta|}\rho^{-q-|\beta|}\\
&\leq M [\rho\theta^b(R-r)]^{-2mp-|\gamma|}(2mp+|\gamma|)!\|u\|_R(R-r)^{-2m}\frac{N\rho}{\varrho-\rho}.
\end{split}
\end{equation*}
Hence
\begin{equation}\label{E:casicaso}
\begin{split}
I_1+I_2+I_3 &\leq M [\rho\theta^b (R-r)]^{-2mp-|\gamma|}(2mp+|\gamma|)!\|u\|_R(R-r)^{-2m}\frac{N\rho}{\varrho-\rho}.
\end{split}
\end{equation}
Lemma \ref{Lema:interpolation}, the induction hypothesis and \eqref{E:casicaso} finish the proof.
\end{proof}

Here we describe the counterexample alluded at the end of Remark \ref{R:1}: let $\omega\subset \Omega$ be an open set and $\varphi\in C_0^\infty(\omega)$, $0\le\varphi\le 1$, with $\varphi\equiv 1$ somewhere in $\omega$.  Define 
\begin{equation*}
V(x,t)=
\begin{cases}
\varphi(x)e^{-\frac{1}{2t-1}},\ &t>\frac12,\\
0,\ & t\le \frac 12,
\end{cases}
\end{equation*}
which is identically zero outside $\omega$ for all times and not time-analytic inside $\omega\times\{\frac 12\}$. Let $u$ be the solution to
\begin{equation*}
\begin{cases}
\p_t u-\Delta u+V(x,t)u=0,\ &\text{in}\ \Omega \times (0,1],\\
u=0,\ &\text{on}\ \p \Omega \times (0,1],\\
u(0)=u_0, &\text{in }\Omega,
\end{cases}
\end{equation*}
with $u_0$ in $C_0^\infty(\Omega)$, $u_0\gneqq 0$ in $\Omega$. The strong maximum principle \cite{Lieberman} shows that $u> 0$ in $\Omega\times (0,1]$ and $e^{t\Delta}u_0$ coincides with $u$ over $\Omega\times [0,\frac 12]$. If $u$ was analytic in the $t$ variable at some point $(x_0,\frac12)$ with $x_0$ in $\Omega$, because all the time derivatives of $u$ and $e^{t\Delta}u_0$ coincide at $(x_0,\frac 12)$, one gets $e^{t\Delta}u_0(x_0,t)=u(x_0,t)$ in $[0,1]$. But $v=u-e^{t\Delta}u_0$ satisfies
\begin{equation*}
\begin{cases}
\p_t v -\Delta v\le 0,\ &\text{in}\ \Omega \times (\frac12,1],\\
v=0,\ &\text{on}\ \p \Omega \times (0,1],\\
v(0)=0,\ &\text{in}\ \Omega,
\end{cases}
\end{equation*}
and the weak maximum principle implies, $v\le 0$ in $\Omega\times [\frac12,1]$. Because $v$ attains its maximum inside $\Omega\times (\frac12,1]$, the strong maximum principle gives, $u=e^{t\Delta}u_0$ in $\Omega\times [0,1]$, which is a contradiction. Thus, $u$ fails to be analytic in the time variable at all points in $\Omega\times\{\frac12\}$.

\section{Observability inequalities}\label{main_observability}
Here we give a proof of the observability inequalities in Theorems \ref{T:otroteoremilla} and \ref{T:otroteoremilla2}. We choose to do it for the equivalent case of the forward parabolic equation. The second parts of the Theorems follow from standard duality arguments and the reasonings in \cite[\S 5]{AEWZ} or \cite[\S 5]{EMZ}.
\begin{proof} From \cite{FursikovOImanuvilov,ImanuvilovYamamoto2} and \eqref{E:1}, the observability inequalities
\begin{align}
&\|u(T)\|_{L^2(\Omega)}\le Ne^{N/\left(1-\e\right)T}\|u\|_{L^2(B_R(x_0)\times (\e T,T))},\label{coeff:space_analytic2}\\
&\|u(T)\|_{L^2(\Omega)}\le Ne^{N/\left(1-\e\right)T}\|\mathbf{A}\nabla u\cdot\nu\|_{L^2(\triangle_R(q_0)\times (\e T,T))},\label{coeff:time_analytic2}
\end{align}
for solutions to
\begin{equation}\label{E: ineterior2}
\begin{cases}
\p_tu-\nabla\cdot\left(\mathbf A\nabla u\right)+\mathbf b_1\cdot\nabla u+\nabla\cdot\left(\mathbf b_2u\right)+cu=0,\ &\text{in}\ \Omega\times (0,T],\\
u=0,\ &\text{in}\ \p\Omega \times [0,T],\\
u(0)=u_0,\ &\text{in}\ \Omega,
\end{cases}
\end{equation}  
with $u_0$ in $L^2(\Omega)$, $0\le\e<1$, $B_{2R}(x_0)\subset\Omega$, $q_0$ in $\partial\Omega$, $\triangle_R(q_0)=B_R(q_0)\cap\partial\Omega$ and $N=N(\Omega, R, \varrho)$, hold when $\partial\Omega$ is $C^{1,1}$ . We may assume that $\mathcal D$ satisfies $|\mathcal D|\ge \varrho|B_R(x_0)|T$ and define
\begin{equation*}
\mathcal D_t=\{x\in\Omega : (x,t)\in\mathcal D\}\quad\text{and}\quad E=\{t\in (0,T): |\mathcal D_t|\ge |\mathcal D|/\left(2T\right)\}.
\end{equation*}
By Fubini's theorem, $\mathcal D_t$ is measurable for a.e. $0<t<T$, $E$ is measurable in $(0,T)$ with $|E|\ge \varrho T/2$. Next, let $z>1$ to be determined later and $0< l<T$ be a Lebesgue point of $E$. From \cite[Lemma 2]{AEWZ}, there is a monotone decreasing sequence $\{l_k\}_{k\ge 1}$, $l<\dots< l_{k+1}<l_l<\dots< l_1\le T$, such that
\begin{equation}\label{E: otracosillamas}
l_{k}-l_{k+1}= z\left(l_{k+1}-l_{k+2}\right)\ \text{and}\ |E\cap (l_{k+1}, l_k)|\ge \tfrac 13 \left(l_k-l_{k+1}\right),\ \text{for}\ k\ge 1.
\end{equation}
Define $\tau_k=l_{k+1}+\tfrac16\left(l_k-l_{k+1}\right)$. From \eqref{coeff:space_analytic2},
\begin{equation}\label{E;quejodidadesi}
\|u(l_k)\|_{L^2(\Omega)}\le Ne^{N/\left(l_k-l_{k+1}\right)}\|u\|_{L^2(B_R(x_0)\times (\tau_k,l_k))},
\end{equation}
Theorem \ref{lagrananaliticidad} shows that the solution $u$ to \eqref{E: ineterior2} verifies 
\begin{equation}\label{E: que agradable}
|\p_x^\alpha \p_t^p u(x,t)|\leq e^{N/\left(l_k-l_{k+1}\right)}\rho^{-1-|\alpha|-p}R^{-|\alpha|} \left(l_k-l_{k+1}\right)^{-p}|\alpha|!p!\,  \|u(l_{k+1})\|_{L^2(\Omega)},
\end{equation}
for $\alpha\in \N^n$, $p\in \N$, $x$ in $B_{R}(x_0)$ and $\tau_k\le t\le l_k$.  
  Then, from \eqref{E;quejodidadesi}, \eqref{E: que agradable} and two consecutive applications of Lemma \ref{propagation}, the first with respect to the time-variable and the second with respect to the space-variables, show that 
 \begin{equation*}
 \|u(l_{k})\|_{L^2(\Omega)}\le \left(Ne^{N/\left(l_k-l_{k+1}\right)}\int_{E\cap \left(l_{k+1},l_k\right)}\|u(t)\|_{L^1(\mathcal D_t)}\,dt\right)^\theta\|u(l_{k+1})\|_{L^2(\Omega)}^{1-\theta},
 \end{equation*}
 holds for any choice of $z>1$ and $k\ge 1$, with $N=N(\Omega, R,\varrho)$, $0<\theta<1$ and $\theta=\theta(\varrho)$. Proceeding with the {\it telescoping series method}, the later implies
\begin{multline*}
\e^{1-\theta} e^{-N/\left(l_k-l_{k+1}\right)} \|u(l_k)\|_{L^2(\Omega)}- \e\, e^{-N/\left(l_k-l_{k+1}\right)}\|u(l_{k+1})\|_{L^2(\Omega)}\\
\le N\int_{E\cap\left(l_{k+1},l_k\right)}\|u(t)\|_{L^1(\mathcal D_t)}\,dt,\;\;\mbox{when}\;\;\e>0.
\end{multline*}
Choosing $\e= e^{-1/\left(l_k-l_{k+1}\right)}$ and \eqref{E: otracosillamas} yield
\begin{equation*}\label{610c1}
\begin{split}
&e^{-\frac{N+1-\theta}{l_k-l_{k+1}}}\|u(l_k)\|_{L^2(\Omega)}- e^{-\frac{N+1-\theta}{l_{k+1}-l_{k+2}}}\|u(l_{k+1})\|_{L^2(\Omega)}\\
&\le N\int_{E\cap\left(l_{k+1},l_k\right)}\|u(t)\|_{L^1(\mathcal{D}_t)}
\,dt,\ \text{when}\ z= \tfrac{N+1}{N+1-\theta}.
\end{split}
\end{equation*}
The addition of the above telescoping series and the local energy inequality for solutions to \eqref{E: ineterior2} leads to
\begin{equation*}
\|u(T)\|_{L^2(\Omega)}\le N\|u\|_{L^1(\mathcal D)},
\end{equation*}
with $N=N(\Omega, T,\mathcal D,\varrho)$.

Similarly, we may assume that $|\mathcal J|\ge \varrho|\triangle_R(q_0)|T$ and setting
\begin{equation*}
\mathcal J_t=\{q\in\partial\Omega : (q,t)\in\mathcal J\}\quad\text{and}\quad E=\{t\in (0,T): |\mathcal J_t|\ge |\mathcal J|/\left(2T\right)\},
\end{equation*}
we get from \eqref{coeff:time_analytic2}, Theorem \ref{lagrananaliticidad} with $x_0=q_0$ and the obvious generalization of Lemma \ref{propagation} for the case of analytic functions defined over analytic hypersurfaces in $\Rn$ that
\begin{equation*}\label{E:un jiodpasomas2}
 \|u(l_{k})\|_{L^2(\Omega)}\le \left(Ne^{N/\left(l_k-l_{k+1}\right)}\int_{E\cap \left(l_{k+1},l_k\right)}\|\mathbf A\nabla u(t)\cdot\nu\|_{L^1(\mathcal J_t)}\,dt\right)^\theta\|u(l_{k+1})\|_{L^2(\Omega)}^{1-\theta}
 \end{equation*}
 for all $k\ge 0$, $z>1$, with $N=N(\Omega, R,\varrho)$, $0<\theta<1$ and $\theta=\theta(\varrho)$. Again, after choosing $z>1$, the telescoping series method implies
 \begin{equation*}
\|u(T)\|_{L^2(\Omega)}\le N\|\mathbf A\nabla u\cdot \nu\|_{L^1(\mathcal J)},
\end{equation*}
with $N=N(\Omega, T, \mathcal J, \varrho)$.
\end{proof}
Finally, Remark \ref{R: elultimo} holds because under \eqref{E:1} with $\mathbf b_2\equiv 0$, the Carleman inequalities and reasonings in \cite{EF} and \cite[\S 3]{EFV} can be used to prove the following global interpolation inequality: there are $N=N(\Omega, R, \varrho)$ and $0<\theta<1$, $\theta=\theta(\Omega, R, \varrho)$ such that
 \begin{equation}\label{E:un jiodpasomas3}
 \|u(t)\|_{L^2(\Omega)}\le \left(Ne^{N/\left(t-s\right)}\|u(t)\|_{L^1(B_R(x_0))}\right)^\theta\|u(s)\|_{L^2(\Omega)}^{1-\theta},
 \end{equation}
 holds, when $0\le s<t\le 1$ and $u$ satisfies \eqref{E: ineterior2}. Also, from Remark \ref{R:1} the solution $u$ to \eqref{E: ineterior2} verifies 
\begin{equation}\label{E: que agradable2}
|\p_x^\alpha u(x,t)|\leq e^{N/\left(l_k-l_{k+1}\right)}\rho^{-1-|\alpha|}R^{-|\alpha|}|\alpha|!\,  \|u(l_{k+1})\|_{L^2(\Omega)},
\end{equation}
for $\alpha\in \N^n$, $x$ in $B_{R}(x_0)$ and $\tau_k\le t\le l_k$. Then, replace respectively \eqref{coeff:space_analytic2} and \eqref{E: que agradable} by \eqref{E:un jiodpasomas3} and \eqref{E: que agradable2} in the proof of Theorem \ref{T:otroteoremilla}.

\medskip

\section{Historical remarks and comments}\label{comment}

It is worth mentioning that the main result Theorem \ref{lagrananaliticidad} in this paper has not been indicated 
in the existing literature related to analyticity properties of solutions to parabolic equations. 
It is motivated by the problem of establishing the null-controllability over measurable sets with positive measure and to obtain the bang-bang property of optimal controls for  general parabolic evolutions.

With the purpose to extend the estimates of the form \eqref{estimatescontrol} to {\it time-dependent} parabolic evolutions, we studied the literature concerned with analyticity properties of solutions to parabolic equations and found the following: most of the works \cite{Friedman3,Friedman2,Tanabe,Friedman1,Eidelman,KinderlehrerNirenberg,Komatsu,Takac0,Takac} make no precise claims about lower bounds for the radius of convergence of the spatial Taylor series of the solutions for small values of the time-variable; the authors were likely more interested in the qualitative behavior.

 If one digs into the proofs, one finds the following: \cite{Friedman3} considers local in space interior analytic estimates for linear parabolic equations and finds a lower bound comparable to $t$. \cite{Friedman2} is a continuation of \cite{Friedman3} for quasi-linear parabolic equations and contains claims but no proofs. The results are based on \cite{Friedman3}. Of course, one can after the rescaling of the local results in \cite{Friedman3} for the growth of the spatial-derivatives over $B_1\times [\frac 12,1]$ for solutions living in $B_2\times (0,1]$, to derive the bound \eqref{E: kimalito} for the spatial directions. \cite{Tanabe} finds a lower bound comparable to $t$.  \cite[ch. 3, Lemma 3.2]{Friedman1} gets close to make a claim like \eqref{estimatescontrol} but the proof and claim in the cited Lemma are not correct, as the inequalities (3.5), (3.6) in the Lemma and the last paragraph in \cite[ch. 3, \S 3]{Friedman1} show when comparing them with the following fact: an exponential factor of the form $e^{1/\rho t^{1/(2m-1)}}$ in the right hand side of \eqref{estimatescontrol} is necessary and should also appear in the right hand side of the inequality (3.6) of the Lemma, for the Gaussian kernel, $G(x,t+\epsilon)$, $t\ge 0$, satisfies $G(iy,2\e)=\left(2\e\right)^{-\frac n2}e^{y^2/8\e}$ and (3.5) in the Lemma independently of $\e>0$, but the conclusion (3.6) in the Lemma  would bound $G(iy,2\e)$, for $y$ small and independently of $\e>0$, by a fixed negative power of $\e$, which is impossible.  The approach in \cite[ch. 3, Lemma 3.2]{Friedman1}, which only uses the existence of the solution over the time interval $[t/2, t]$ to bound all the derivatives at time $t$, cannot see the exponential factor and find a lower bound for the spatial radius of convergence independent of $t$. On the contrary, the methods in \cite[ch. 3]{Friedman1} are easily seen to imply \eqref{E: kimalito}. \cite{KinderlehrerNirenberg} and \cite{Komatsu} deal with non-linear parabolic second order evolutions and find a lower bound comparable to $t$. \cite{Takac0,Takac} consider linear problems and find a lower bound comparable to $t^{\frac 1{2m}+\e}$, for all $\e>0$. See also \cite[\S 6]{Takac0} and \cite[\S 9]{Takac} for a historical discussion.

Finally, \cite[p. 178 Th. 8.1 (15)]{Eidelman} builds a holomorphic extension in the space-variables of the fundamental solution for high-order parabolic equations or systems. This holomorphic extension is built upon the assumptions of local analyticity of the coefficients in the spatial-variables and continuity in the time-variable. The later provides an alternative proof of \eqref{estimatescontrol} with $p=0$ at points in the interior of $\Omega$. 
As far as we know, Eidelman's School did not work out similar estimates for the complex holomorphic extension of the Green's function with zero lateral Dirichlet conditions for $\mathcal L$ over $\Omega$ up to the boundary. If they had done so, it would provide another proof of (1.7) up to the boundary. We believe that such approach is more complex than the one in this work.

On the other hand, the motivation to prove the estimates of the form \eqref{estimatescontrol} comes from its applications to the null-controllability of parabolic evolutions with bounded controls acting over measurable sets of positive measure. To describe these results we begin with a report of the progresses made on the null-controllability and observability of parabolic evolutions over measurable sets. In what follows, $\omega$, $\gamma$ and $E$ denote subsets of $\Omega$, $\partial\Omega$ and  $(0,T)$ respectively: except for the 1997 work \cite{MizelSeidman} - where the authors proved the one-sided boundary observability of the heat equation in one space dimension over measurable sets - up to 2008 the control regions considered in the literature were always of the type $\omega\times (0,T)$ or $\gamma\times (0,T)$, with $\omega$ and $\gamma$ open. Then, \cite{W1} showed that the heat equation is observable over sets $\omega\times E$, with $\omega$ open and $E$ measurable with positive measure. \cite{AE} showed that second order parabolic equations with {\it time-independent} Lipschitz coefficients associated to {\it self-adjoint} elliptic operators with local analytic coefficients in a neighborhood of a measurable set with positive measure $\omega$ are observable over $\omega\times (0,T)$; and that the same holds for one dimensional parabolic operators with {\it time-independent} measurable coefficients. Both \cite{W1} and \cite{AE} relied on the Lebeau-Robbiano  strategy \cite{G. LebeauL. Robbiano} for the construction of control functions. \cite{C1} combined the reasonings of \cite{W1} and \cite{AE} to obtain the observability of the heat equation over arbitrary cartesian products of measurable sets $\omega\times E$ with positive measure. \cite{PW1} and \cite{PWZ} showed the observability of $\partial_t-\Delta+c(x,t)$, with $c$ a bounded function, over sets $\omega\times E$ with $\omega$ open and $E$ measurable with positive measure. These two works used Poon's parabolic frequency function \cite{Poon}, its further developments in \cite{EFV} and the telescoping series method \cite{Miller2}. \cite{AEWZ} established the interior and boundary null-controllability with bounded controls of the heat equation over general measurable sets $\mathcal D\subset\Omega\times (0,T)$ and $\mathcal J\subset\partial\Omega\times(0,T)$ with positive measure.  Finally, \cite{EMZ} extended the results in \cite{AEWZ} to higher order parabolic evolutions or systems with {\it time-independent} coefficients associated to possibly {\it non self-adjoint} elliptic operators with global analytic coefficients when $\partial\Omega$ is analytic.

\section{Appendix} \label{S:apendice}
Here we prove the weighted $L^2$ estimates we need in Section \ref{main_analyticity}. To prove them we use the standard $W^{2m,1}_2$ Schauder estimates.
\begin{lemma}\label{Lema:interpolation}
Let $0\le\theta\le 1$ and $\Omega$ be a Lipschitz domain. Then, there is $N=N(m,n,\Omega)$ such that
\begin{multline}\label{interpolation}
\|t^{p+\frac{k}{2m}}\peso D^k u \|\\ \leq N\left[ \|t^p\peso u\|^{\frac{2m-k}{2m}}\|t^{p+1}\peso D^{2m}u\|^{\frac{k}{2m}} +\|t^p\peso u\|\right]
\end{multline}
holds for all $k=1,\ldots,2m-1$, $p\geq 0$ and $u$ in $C^{\infty}(\overline{\Omega}\times [0,1])$.
\end{lemma}
\begin{remark}\label{R: otraapreciacion} When $\Omega$ is either $B_R$ or $B_R^+$, $R>0$, then 
\begin{multline}
\|t^{p+\frac{k}{2m}}\peso D^k u \|_{L^2(\Omega\times (0,1))}\\ \leq N\left[ \|t^p\peso u\|_{L^2(\Omega\times (0,1))}^{\frac{2m-k}{2m}}\|t^{p+1}\peso D^{2m}u\|_{L^2(\Omega\times (0,1))}^{\frac{k}{2m}} \right.
\\\left.+R^{-k}\|t^p\peso u\|_{L^2(\Omega\times (0,1))}\right],
\end{multline}
with $N=N(m,n)$.
\end{remark}
\begin{proof}
By the interpolation inequality \cite[Theorems 4.14, 4.15]{Adams}, there is $N=N(m,\Omega)$ such that
\begin{equation}\label{9}
\|D^k u(t)\|_{L^2(\Omega)}\leq N \left[\|u(t)\|_{L^2(\Omega)}^{\frac{2m-k}{2m}} \|D^{2m}u(t)\|_{L^2(\Omega)}^{\frac{k}{2m}}+\|u(t)\|_{L^2(\Omega)}\right],
\end{equation}
when $1\le k< 2m$. Now, multiply \eqref{9} by $t^{p+\frac{k}{2m}}\peso$ and H\"{o}lder's inequality over $[0,1]$ yields \eqref{interpolation}.
\end{proof}
\begin{lemma}\label{Lema:regularidad1}  Let $u$ in $C^\infty(\overline{\Omega}\times [0,1])$ satisfy 
\begin{equation*}
\begin{cases}
\p_tu+\mathcal L u=F,\ &\text{in}\ \Omega\times (0,1],\\
u=Du=\ldots=D^{m-1}u=0 ,\ &\text{in}\ \partial\Omega\times (0,1].\\
\end{cases}
\end{equation*}
Then, there is $N=N(\Omega,n,\varrho,m)$ such that
\begin{multline} \label{regularidad1}
\|t^{p+1}\peso \partial_t u \|+\sum_{l=0}^{2m}\|t^{p+\frac l{2m}}\peso D^{l} u \| \\
\leq N\left[ (p+k+1) \|t^p e^{{-\frac{k-1}{k}}\theta t^{-\sigma}} u \| +\|t^{p+1}\peso F \|  \right],
\end{multline}
holds for any $\theta\ge 0$, $p\geq0$ and $k\geq 2$.
\end{lemma}
\begin{proof}
Define $v=t^{p+1}\peso u$, then $v$ satisfies $\p_t v+ \mathcal L v= G$ in $\Omega\times (0,1]$, with
\begin{equation}\label{E:quies es}
G=t^{p+1}\peso F+\left[(p+1)t^p \peso +\sigma \theta t^{p-\sigma}\peso\right]u.
\end{equation}
For $t>0$ and $k\ge 2$,
\begin{equation}\label{unaacotacion}
\theta t^{p-\sigma} \peso =  \tfrac\theta k t^{-\sigma}e^{-\frac\theta k t^{-\sigma}}ke^{-\theta\frac{k-1}{k}t^{-\sigma}}\le k e^{-\theta \frac{k-1}{k}t^{-\sigma}}.
\end{equation}
By the $W^{2m,1}_{2}$ Schauder estimate \eqref{E: Schauder estimate},
\begin{equation} \label{CZ1}
\|\p_t v\|+\|D^{2m} v\|\leq N\left[ \|v\| +\|G\|\right],
\end{equation}
with $N=N(\Omega,n,\varrho,m)$ and \eqref{regularidad1} follows from \eqref{CZ1}, \eqref{unaacotacion}, \eqref{E:quies es} and Lemma \ref{Lema:interpolation}.
\end{proof}
Lemma \ref{Lema:regularidad2} is a well-known  estimate near the boundary. It can be found in \cite[Theorem 7.22]{Lieberman} for $m=1$. We prove it here for completeness.
\begin{lemma} \label{Lema:regularidad2}
Let $u$ in $C^{\infty}(B_R^+\times [0,1])$ verify
\begin{equation*}
\begin{cases}
\p_tu+\mathcal Lu=F,\ &\text{in}\ B_R^+\times(0,1],\\
u=Du=\ldots=D^{m-1}u=0 ,\ &\text{in}\ \{x_n=0\}\cap\partial B_R^+\times(0,1],\\
u(0)=0 ,\ &\text{in}\ \ B_R^+.
\end{cases}
\end{equation*}
and $0<r<r+\delta<R\le 1$. Then, there is $N=N(n,\varrho,m)$ such that
\begin{equation} \label{regularidad2}
\|\partial_t u \|_r'+\|D^{2m} u \|_r' 
\leq N\left[\delta^{-2m} \| u \|_{r+\delta}' +\| F \|_{r+\delta}'  \right].
\end{equation}
\end{lemma}
\begin{proof}
Let $\eta$ in $C_0^\infty(B_R)$ be such that for $0<\lambda<1$
\begin{equation*}
\eta(x)=
\begin{cases}
1,\ &\text{in\ }B_{r+\lambda \delta},\\
0,\ &\text{in\ }B_{r+\frac{1+\lambda}{2}\delta}^c,\\
\end{cases}
\end{equation*}
and $|D^k \eta|\leq C_m \left[(1-\lambda) \delta\right]^{-k}$, for $k=0,\ldots,2m$. Define $v=u \eta$, then
\begin{equation*}
\p_tv +\mathcal Lv=\eta F+ (-1)^m\sum_{|\alpha|\leq 2m}a_\alpha \sum_{\gamma < \alpha} \binom{\alpha}{\gamma} \p_x^{\alpha-\gamma} \eta \p_x^{\gamma} u.
\end{equation*}
By the $W^{2m,1}_{2}$ Schauder estimate over $B_R^+\times (0,T]$ applied to $v$ \cite[Theorem 4]{DongKim}
\begin{multline}\label{CZ2}
\|\p_t u\|_{r}'+\|D^{2m}u\|_{r+\lambda \delta}'\\ \leq N \left[\|F\|_{r+\frac{1+\lambda}{2}\delta}'+\sum_{k=0}^{2m-1}\left[(1-\lambda)\delta\right]^{k-2m}\|D^ku\|_{r+\frac{1+\lambda}{2}\delta}'  \right].
\end{multline}
Define the semi-norms
\begin{align*}
|u|_{k,\delta}&=\sup_{\mu\in (0,1)}{\left[(1-\mu)\delta \right]^k \|D^k u\|_{r+\mu \delta}'},\ k=0,\ldots,2m.
\end{align*}
Estimate \eqref{CZ2} can be rewritten in terms of these semi-norms as follows
\begin{equation} \label{enseminormas}
\delta^{2m}\|\p_t u\|_{r}'+|u|_{2m,\delta}\leq N\left[\sum_{k=0}^{2m-1} |u|_{k,\delta}+\delta^{2m}\|F\|_{r+\delta}' \right].
\end{equation}
To eliminate the terms $|u|_{k,\delta}$ from the right hand side of \eqref{enseminormas}, recall that the semi-norms interpolate (\cite[Theorem 4.14]{Adams} and \cite[p. 237]{GilbargTrudinger}); i.e., there is $c=c(n,m)$ such that
$$|u|_{k,\delta}\leq \e |u|_{2m,\delta}+c\e^{-\frac{k}{2m-k}}\|u\|_{r+\delta}',$$
for any $\e\in(0,1)$, so
$$\sum_{k=0}^{2m-1} |u|_{k,\delta}\leq 2m\e|u|_{2m,\delta} + c \sum_{k=0}^{2m-1}\e^{-\frac{k}{2m-k}}\|u\|_{r+\delta}'.$$
Choose then $\e\le\frac{1}{4mN}$ and from \eqref{enseminormas}
\begin{equation*}
\delta^{2m}\|\p_t u\|_r'+|u|_{2m,\delta}\leq N\left[\|u\|_{r+\delta}' +\delta^{2m}\|F\|_{r+\delta}'\right],
\end{equation*}
which yields \eqref{regularidad2}.
\end{proof}

Lemma \ref{Lema:regularidad2_interior} is the interior analogue of Lemma \ref{Lema:regularidad2} but now using \cite[Theorem 2]{DongKim}.
\begin{lemma} \label{Lema:regularidad2_interior}
Let $u$ in $C^{\infty}(B_R\times [0,1])$ verify
\begin{equation*}
\begin{cases}
\p_tu+\mathcal Lu=F,\ &\text{in}\ B_R\times(0,1],\\
u(0)=0 ,\ &\text{in}\ B_R,
\end{cases}
\end{equation*}
and $0<r<r+\delta<R\le 1$. Then, there is $N=N(n,\varrho,m)$ such that
\begin{equation*}
\|\partial_t u \|_r+\|D^{2m} u \|_r 
\leq N\left[\delta^{-2m} \| u \|_{r+\delta} +\| F \|_{r+\delta}  \right].
\end{equation*}
\end{lemma}
Lemmas \ref{Lema:regularidad2} and \ref{Lema:regularidad1} imply Lemma \ref{Lema:regularidad3}.

\begin{lemma} \label{Lema:regularidad3}
Let $u$ in $C^{\infty}(B_R^+\times [0,1])$ satisfy
\begin{equation*}
\begin{cases}
\p_tu+ \mathcal Lu=F,\ &\text{in}\ B_R^+\times(0,1],\\
u=Du=\ldots=D^{m-1}u=0 ,\ &\text{in}\ \{x_n=0\}\cap\partial B_R^+\times(0,1]\\
\end{cases}
\end{equation*}
and $0<r<r+\delta<R\le 1$. Then, there is $N=N(n,\varrho,m)$ such that
\begin{multline*}
\|t^{p+1}\peso \p_t u \|_r'+\|t^{p+1}\peso D^{2m} u \|_r' 
\\ \leq N\left[ (p+k) \|t^p e^{{-\frac{k-1}{k}}\theta t^{-\sigma}} u \|_{r+\delta}'\right.\\\left.+\delta^{-2m}\|t^{p+1} \peso u\|_{r+\delta}' +\|t^{p+1}\peso F \|_{r+\delta}'   \right],
\end{multline*}
for $0<\theta\le 1$, $p\geq 0$ and $k\geq 2$.
\end{lemma}

Similarly, Lemmas  \ref{Lema:regularidad1} and \ref{Lema:regularidad2_interior} imply the weighted $L^2$ estimate in Lemma \ref{Lema:regularidad4}.
\begin{lemma} \label{Lema:regularidad4}
Let $u$ in $C^\infty(B_R\times [0,1])$ satisfy
$$\p_tu+ \mathcal Lu=F\ \text{ in }\ B_R \times (0,1]$$
and $0<r<r+\delta<R\le 1$. Then there is $N=N(n,\varrho,m)$ such that
\begin{multline*}
\|t^{p+1}\peso \p_t u \|_r+\|t^{p+1}\peso D^{2m} u \|_r 
\\ \leq N\left[ (p+k) \|t^p e^{{-\frac{k-1}{k}}\theta t^{-\sigma}} u \|_{r+\delta}\right.\\\left.+\delta^{-2m}\|t^{p+1} \peso u\|_{r+\delta} +\|t^{p+1}\peso F \|_{r+\delta}   \right],
\end{multline*}
holds for $0<\theta\le 1$, $p\geq 0$ and $k\geq 2$. 
\end{lemma}
\begin{lemma}\label{L:postleibniz} If $\gamma\in \N^n$, $0<t<s$, 
\begin{equation*}
\sum_{\beta<\gamma}\binom{\gamma}{\beta}|\gamma-\beta|!|\beta|!s^{-|\gamma|+|\beta|} t^{-|\beta|}\leq |\gamma|!\, \frac{ t^{1-|\gamma|}}{s-t}.
\end{equation*}
\end{lemma}
\begin{proof} 
Let $f(x)=\varphi(u)$, with $u=(x_{1}+\cdots+x_{n})$ and $\varphi(u)=(1-u)^{-1}$.
Then, $\frac{\partial^{\gamma}}{\partial x^{\gamma}}f(x)=\varphi^{|\gamma|}(u)=|\gamma|!u^{-|\gamma|-1}$.
Now let, $f_{t}(x)=f(\frac{x}{t})$, $\frac{\partial^{\gamma}}{\partial x^{\gamma}}f_{t}(x)=t^{-|\gamma|}\frac{\partial^{\gamma}}{\partial x^{\gamma}}f(\frac{x}{t})$, and taking $x=0$, we have $\frac{\partial^{\gamma}}{\partial x^{\gamma}}f_{t}(0)=|\gamma|!t^{-|\gamma|}$. Now, set $g(x)=f_{s}(x)f_{t}(x)=\psi(u)$, with
\begin{equation*}
\psi(u)=\frac{1}{(1-\frac{u}{s})(1-\frac{u}{t})}\, .
\end{equation*}
Let $|u|<t$. Then
\begin{equation*}
\psi(u)=\sum_{i=0}^{+\infty}\left(u/s\right)^{i}\sum_{j=0}^{+\infty}\left(u/t\right)^{j}=\sum_{i,j=0}^{+\infty}\frac{u^{i+j}}{s^{i}t^{j}}=
\sum_{k=0}^{+\infty}u^{k}\sum_{i+j=k}\frac{1}{s^{i}t^{j}}
\end{equation*}
and
\begin{equation*}
\psi^{(k)}(0)=k!\,t^{-k}\sum_{i=0}^{k}\left(t/s\right)^{i},\ \text{for}\ k\ge 0.
\end{equation*}
Thus,
\begin{equation*}
\frac{\partial^{\gamma}g}{\partial x^{\gamma}}(0)=\psi^{(|\gamma|)}(0)=|\gamma|!\,t^{-|\gamma|}\sum_{i=0}^{|\gamma|}\left(t/s\right)^{i},\ \text{for}\ \gamma\in\N^n.
\end{equation*}
From Leibniz's rule
\begin{align*}
\frac{\partial^{\gamma}g}{\partial x^{\gamma}}(0)&=\sum_{\beta\leq\gamma}\binom{\gamma}{\beta}\,\partial^{\gamma-\beta}f_{s}(0)\partial^{\beta}f_{t}(0)\\
&=\sum_{\beta\leq\gamma}\binom{\gamma}{\beta}|\gamma-\beta|!\,|\beta|!\,s^{-|\gamma|+|\beta|}t^{-|\beta|}.
\end{align*}
It implies that
\begin{equation*}
\sum_{\beta\leq\gamma}\binom{\gamma}{\beta}|\gamma-\beta|!|\beta|!s^{-|\gamma|+|\beta|}t^{-|\beta|}=|\gamma|!\,t^{-|\gamma|}\sum_{i=0}^{|\gamma|}\left(\frac{t}{s}\right)^{i},
\end{equation*}
where dropping the term corresponding to $\beta=\gamma$,
\begin{equation*}
\begin{split}
\sum_{\beta<\gamma}\binom{\gamma}{\beta}|\gamma-\beta|!\,|\beta|!\,s^{-|\gamma|+|\beta|}t^{-|\beta|}=
|\gamma|!\,t^{-|\gamma|}\left(\sum_{i=0}^{|\gamma|}\left(\frac{t}{s}\right)^{i}\right)-|\gamma|!\,t^{-|\gamma|}\\
=|\gamma|!\,t^{-|\gamma|}\sum_{i=1}^{|\gamma|}\left(\frac{t}{s}\right)^{i}\leq|\gamma|!\,t^{-|\gamma|}\sum_{i=1}^{+\infty}\left(\frac{t}{s}\right)^{i}=
|\gamma|!\,\frac{t^{1-|\gamma|}}{s-t}\ ,
\end{split}
\end{equation*}
if $0<t<s$.
\end{proof}
\begin{lemma}\label{L:ultimo} Let $\mathcal{L}=(-1)^m\sum_{|\alpha|\leq 2m}a_\alpha(x)\p^\alpha$ be an elliptic operator in non-variational form whose coefficients satisfy \eqref{parabolicidad} and $a_\alpha$ belong to  $C^{|\alpha|-m}$, for $m<|\alpha|\leq 2m$. Then, $\mathcal L$ can be written in variational form as
\begin{equation*}
\mathcal L=(-1)^m\sum_{|\gamma|,|\beta|\leq m}\p^\gamma(A_{\gamma\beta}(x) \p^\beta),
\end{equation*}
with
\begin{equation*}\label{parabolicidad42}
\sum_{|\gamma|,|\beta|\le m}\|A_{\gamma\beta}\|_{L^\infty(\Omega)}\le\varrho^{-1},\quad
\sum_{|\gamma|=|\beta|=m} A_{\gamma\beta}(x)\xi^{\gamma}\xi^{\beta}\ge \varrho |\xi|^{2m},
\end{equation*}
for $x$ in $\overline\Omega$ and $\xi$ in $\R^n$.  
\end{lemma}
\begin{proof}
We write
\begin{equation*}
\sum_{|\alpha|\leq 2m}a_\alpha \p^\alpha=\sum_{|\alpha|\leq m}a_\alpha \p^\alpha+\sum_{m+1\leq|\alpha|\leq 2m}a_\alpha \p^\alpha\triangleq I+J.
\end{equation*}
Then
$$I=\sum_{|\alpha|\leq m}a_\alpha \p^\alpha=\sum_{|\gamma|,|\beta|\leq m}\p^\gamma(A_{\gamma\beta} \p^\beta), $$
with $A_{\gamma\beta}=0$ if $|\gamma|\neq 0$ and $A_{\gamma\beta}=a_\beta$, if $\gamma=0$; and $I$ can be written in non-variational form. For each $\alpha\in \N^n$ with $m+1\le |\alpha|\le 2m$, define 
\begin{equation*}
c_\alpha=\#\lbrace (\gamma,\beta):\ \gamma+\beta=\alpha\ \text{and}\ |\beta|=m \rbrace.
\end{equation*}
Then,
\begin{equation*}
J=\sum_{j=m+1}^{2m} \sum_{|\gamma|=j-m,|\beta|=m}\frac{1}{c_{\gamma+\beta}}a_{\gamma+\beta}\p^{\gamma+\beta}\triangleq\sum_{j=m+1}^{2m} \sum_{|\gamma|=j-m,|\beta|=m}A_{\gamma\beta}\p^{\gamma+\beta}.
\end{equation*}
Since $A_{\gamma\beta}\in C^{|\gamma|+|\beta|-m}$, we can apply Lemma \ref{lemmapuñetero} below to each of the  terms $A_{\gamma\beta}\p^{\gamma+\beta}$ and we are done. It remains to check that the new expression satisfies the ellipticity condition: we notice that 
\begin{equation*}\label{igualdadoperadores}
\sum_{|\alpha|=  2m}a_\alpha(x_0)\p^\alpha=\sum_{|\gamma|=|\beta|=  m}A_{\gamma\beta}(x_0)\p^{\gamma+\beta},
\end{equation*} 
for each fixed $x_0$ in $\overline\Omega$. Therefore, if we apply these constant coefficients operators to any rapidly decreasing function and take Fourier transform, we get 
$$\sum_{|\alpha|=  2m}a_\alpha(x_0)\xi^\alpha=\sum_{|\gamma|=|\beta|=  m}A_{\gamma\beta}(x_0)\xi^{\beta+\gamma},\ \ \ \text{for any }\xi\in \R^n,$$
which implies the uniform ellipticity of the operator in variational form, when the operator in non-variational form is uniformly elliptic.
\end{proof}
\begin{lemma}\label{lemmapuñetero}
For $m\in \N$ and $j=m+1,\ldots,2m$ the following holds:
if $a\in C^{j-m}$ and $\alpha\in \N^n$ is a multi-index with $|\alpha|\leq j$,  we can write
\begin{equation}\label{formavariacional}
a\p^\alpha=\sum_{|\gamma|,|\beta|\leq m}\p^\gamma(b_{\gamma\beta}\p^\beta)
\end{equation}
for some $b_{\gamma\beta}\in C^{|\gamma|}$.
\end{lemma}
\begin{proof}
For each fixed $m\in \N$ we prove Lemma \ref{lemmapuñetero} by induction on $j$.

\textit{Case $j=m+1$:} if $|\alpha|\leq m$, $a\p^\alpha$ already has the required form. Then, we only need to check the statement for multi-indices $\alpha\in \N^n$ with $|\alpha|=m+1$. If $j=m+1$ and $a\in C^1$, $\alpha=e_i+\beta$ for some multi-index $\beta$ with $|\beta|=m$. Then,
$$a\p^\alpha=\p^{e_i}(a\p^\beta)-\p^{e_i}a\p^\beta,$$
which already has the required form because $|e_i|$ and $|\beta|\leq m$.

Now, assuming that the statement holds for $j=m+1,\ldots,k$, we prove it for $j=k+1$.
Let $a\in C^{k+1-m}$. If $\alpha\in \N^n$ with $|\alpha|\leq k$, the induction hypothesis shows that $a\p^\alpha$ can be written as \eqref{formavariacional}; hence, we only need to check that the statement holds for multi-indices $\alpha$ with $|\alpha|=k+1$. In this case, we can write $\alpha=\gamma+\beta$ with $|\gamma|=k+1-m$ and $|\beta|=m$; then, applying Leibniz's rule we get
\begin{equation}\label{ultimopaso}
a\p^\alpha=a\p^{\gamma+\beta}=\p^\gamma(a\p^\beta)-\sum_{0<\sigma\leq \gamma}\binom{\gamma}{\sigma}\p^\sigma a \p^{\beta+\gamma-\sigma}
\end{equation}
and we want to apply the induction hypothesis to  each of the terms of the sum in \eqref{ultimopaso}. Thus we only need to check that each operator $\p^\sigma a \p^{\beta+\gamma-\sigma}$ satisfies hypothesis which fall into one of the previous steps of the induction hypothesis:
\begin{itemize}
\item Because $a\in C^{k+1-m}$, we have $\p^\sigma a\in C^{k+1-|\sigma|-m}$.
\item Because the sum in \eqref{ultimopaso} runs for multi-indices $\sigma$ with $1\leq |\sigma|=|\gamma|\leq k+1-m$, we have that $|\beta+\gamma-\sigma|=k+1-|\sigma|\leq k$.
\end{itemize}
This finishes the proof of Lemma \ref{lemmapuñetero}.
\end{proof}

\end{document}